\documentclass[ leqno]{amsart}
\usepackage{amssymb}
\usepackage{mathrsfs}
\usepackage{amsmath, amsfonts, vmargin, enumerate}
\usepackage{graphics}
\usepackage{verbatim}
\usepackage{color}
\usepackage{amsthm}
\usepackage{latexsym,bm}
\usepackage{euscript}
\usepackage{color}
\usepackage[colorlinks=true,citecolor=blue,linkcolor=blue]{hyperref}

\input xypic
\xyoption {all}

 \makeatletter

\newtheorem{thm}{Theorem}[section]
\newtheorem{prop}[thm]{Proposition}
\newtheorem{cor}[thm]{Corollary}
\newtheorem{lem}[thm]{Lemma}
\newtheorem{defi}[thm]{Definition}
\newtheorem{remark}[thm]{Remark}
\newtheorem{example}[thm]{Example}
\newtheorem{pb}[thm]{Problem}

\newenvironment{ex}{\begin{example}\rm}{\end{example}}

\numberwithin{equation}{section}

\newcommand{\real}{{\mathbb R}}
\newcommand{\nat}{{\mathbb N}}
\newcommand{\ent}{{\mathbb Z}}
\newcommand{\com}{{\mathbb C}}
\newcommand{\T}{{\mathbb T}}

\newcommand{\cK}{{\mathcal K}}
\newcommand{\mcL}{{\mathcal L}}
\newcommand{\rTr}{{\mathrm{Tr} }}

\newcommand{\Bc}{\mathcal{B}}

\newcommand{\ra}{{\rightarrow}}
\newcommand{\8}{{\infty}}
\newcommand{\be}{\begin{eqnarray*}}
\newcommand{\ee}{\end{eqnarray*}}
\newcommand{\beq}{\begin{equation}}
\newcommand{\eeq}{\end{equation}}
\newcommand{\beqn}{\begin{equation*}}
\newcommand{\eeqn}{\end{equation*}}
\newcommand{\D}{{D}}

\newcommand{\pd}{{\partial}}
\newcommand{\p}{{\psi}}
\newcommand{\ri}{{\rm{i}}}

\newcommand{\vf}{{\varphi}}

\newcommand{\qt}{{\mathbb{T}_\theta^d}}
\newcommand{\ot}{{\otimes}}
\newcommand{\sgn}{{{\rm{sgn}}}}
\newcommand{\g}{{\gamma}}
\newcommand{\al}{{\alpha}}
\newcommand{\bt}{{\beta}}
\newcommand{\lan}{{\langle}}
\newcommand{\ran}{{\rangle}}

\newcommand{\Jx}{{\langle \xi \rangle}}

\newcommand{\wh}{\widehat}

\newcommand{\tr}{\mathrm{tr}}

\newcommand{\rank}{\mathrm{rank}}

\newcommand{\Cplx}{\mathbb{C}}
\newcommand{\Sp}{\mathbb{S}}
\newcommand{\Sc}{{C^\infty}}
\newcommand{\Dist}{{\mathcal{D}}}

\newcommand{\cS}{{\Sc}}
\def\qd{\,{\mathchar'26\mkern-12mu d}}

\newcommand{\hl}{}

\begin{document}

\title{Quantum differentiability on quantum tori}

\author{Edward MCDONALD}

\address{School of Mathematics and Statistics, UNSW, Kensington, NSW 2052, Australia}
\email{edward.mcdonald@unsw.edu.au}



\author{Fedor SUKOCHEV}

\address{School of Mathematics and Statistics, UNSW, Kensington, NSW 2052, Australia}
\email{f.sukochev@unsw.edu.au}

\author{Xiao XIONG}

\address{Institute for Advanced Study in Mathematics, Harbin Institute of Technology, 150001 Harbin, China, and School of Mathematics and Statistics, UNSW, Kensington, NSW 2052, Australia}
\email{xxiong@hit.edu.cn}

\date{}
\maketitle

\markboth{E. Mcdonald, F. Sukochev and X. Xiong}%
{Quantum differentiability on quantum tori}

\begin{abstract}
    We provide a full characterisation of quantum differentiability (in the sense of Connes) on quantum tori. We also prove
    a quantum integration formula which differs substantially from the commutative case.

\end{abstract}

\section{Introduction}

    Quantum tori (also known as noncommutative tori and irrational rotation algebras) are landmark examples in noncommutative geometry. These algebras have featured in many directions in physics, such as the study of the quantum Hall effect \cite{Bellissard-original, Bellissard-van-Elst-Schulz-Baldes, Xia-qhe}, Matrix theory \cite{Connes-Douglas-Schwarz-matrix-theory-1998}, string theory \cite{Seiberg-Witten} and deformation quantisation \cite{Rieffel-deformation-quantization}.
    Quantum tori have been heavily studied from the perspective of operator algebras \cite{Effros-Hahn-memoirs-1967,Pimsner-Voiculescu-crossed-products-1980, Rieffel-1981} and were later taken as a fundamental example in noncommutative
    geometry (see \cite{Connes1980}, \cite[Chapter 12]{green-book} and \cite{CM2014}). In the context of foliation theory, quantum tori are studied as the $C^*$-algebra associated to a Kronecker foliation \cite[Chapter 2, Section 9.$\beta$]{Connes1994}.    
    
    A. Connes introduced the quantised calculus in \cite{Connes-ncdg-1985} as an analogue of the algebra of differential forms in a noncommutative setting, and later explored the link with the action functional of Yang-Mills theory \cite{Connes1988}.
    Connes successfully applied quantised calculus in computing the Hausdorff measure of Julia sets and limit sets of Quasi-Fuchsian groups in the plane \cite[Chapter 4, Section 3.$\gamma$]{Connes1994} (for a more recent exposition see \cite{CSZ,CMSZ}).
    
    The core ingredients of the quantised calculus, as outlined in \cite{Connes-ncdg-1985}, are a separable Hilbert space $H$, a unitary self-adjoint operator $F$ on $H$ and a $C^*$-algebra $\mathcal{A}$ represented on $H$ such that 
    for all $a \in \mathcal{A}$ the commutator $[F,a]$ is a compact operator on $H$. Then the quantised differential of $a \in \mathcal{A}$ is defined to be the operator $\qd a = \ri[F,a]$. The compact operators
    on $H$ are described by Connes as being analogous to infinitesimals, and the rate of decay of the sequence of singular values:
    \begin{equation*}
        \mu(n,T) := \inf\{\|T-R\|\;:\;\mathrm{rank}(R)\leq n\}
    \end{equation*}
    corresponds in some way to the ``size" of the infinitesimal $T$ (see \cite{Connes1995}). In this setting one can quantify the smoothness of an element $a \in \mathcal{A}$
    in terms of the rate of decay of $\{\mu(n,\qd a)\}_{n=0}^\infty$. Of particular interest are those elements $a \in \mathcal{A}$ which satisfy:
    \begin{align*}
                            \mu(n,\qd a) &= O((n+1)^{-1/p}),\quad  n\to \infty,\text{ or,}\\
        \sum_{n=0}^\infty \mu(n,\qd a)^p &< \infty,\text{ or,}\\
        \sup_{n \geq 1} \frac{1}{\log(n+2)} \sum_{k=0}^n \mu(k,\qd a)^p &< \infty\,,
    \end{align*}
    for some $p \in (0,\infty)$. The first condition stated above is that $\qd a$ is in the weak-Schatten ideal $\mcL_{p,\infty}$, the second condition
    is for $\qd a$ to be in the Schatten ideal $\mcL_p$, and the final condition is that $|\qd a|^p$ is in the Macaev-Dixmier ideal $\mathcal{M}_{1,\infty}$ \cite[Chapter 4, Section 2.$\beta$]{Connes1994} (see also \cite[Example 2.6.10]{LSZ2012}).
    
    The link between quantised calculus and geometry is discussed by Connes in \cite{Connes1988}. A model example for quantised calculus is to take a compact Riemannian spin manifold $M$ with Dirac operator $D$,
    and define $H$ to be the Hilbert space of square integrable sections of the spinor bundle. The algebra $\mathcal{A} = C(M)$ of continuous functions on $M$ acts by pointwise multiplication on $H$, and one defines
    \begin{equation*}
        F := \chi_{[0,\infty)}(D)-\chi_{(-\infty,0)}(D).
    \end{equation*}
    One then has $\qd f = \ri[F,M_f]$, where $M_f$ is the operator on $H$ of pointwise multiplication by $f$.
    In quantised calculus the immediate question is to determine the relationship between the degree of differentiability of $f \in C(M)$ and 
    the rate of decay of the singular values of $\qd f$. In general, we have the following:
    \begin{equation*}
        f \in C^\infty(M) \Rightarrow |\qd f|^d \in \mathcal{M}_{1,\infty},
    \end{equation*}
    where $d$ is the dimension of the manifold $M$ \cite[Theorem 3.1]{Connes1988}.
          
    For certain special cases it is possible to obtain a far more precise understanding of the relationship between the smoothness of $f$ and the singular values of $\qd f$. 
    The simplest example is to take the unit circle $\T = \{z \in \com\;:\; |z| = 1\}$,
    with $\mathcal{A} = C(\T)$, $H = L_2(\T)$ and the standard choice of $F$ in this setting is the Hilbert transform. Then by a result of V. Peller \cite[Theorem 7.3]{Peller2003}, we have that
    for any $p \in (0,\infty)$: $\qd f \in \mcL_p$ if and only if $f$ is in the Besov space $B^{1/p}_{p,p}(\T)$. Peller's work has been extended to obtain even more precise relationships between $f$ and the singular values of $\qd f$, for example 
    L. Gheorghe \cite{Gheorghe2001} found necessary and sufficient conditions on $f$ to ensure that $\qd f$ is in an arbitrary Riesz-Fisher space. For more details from a quantised calculus perspective, see \cite[Chapter 4, Section 3.$\alpha$]{Connes1994}.
    
    In higher dimensions, the relationship between $f$ and $\qd f$ has also been studied \cite{JW1982,RS1989,CST1994}. To illustrate the situation, consider the $d$-dimensional torus $\T^d${\hl, $d \geq 2$}. The appropriate Dirac operator in this setting
    is:
    \begin{equation*}
        D = \sum_{j=1}^d -\ri\gamma_j\otimes \partial_j,
    \end{equation*}
    where $\partial_j$ denotes differentiation with respect to the $j$th coordinate on $\T^d$, and $\{\gamma_1,\ldots,\gamma_d\}$ denotes the $d$-dimensional Euclidean gamma matrices,
    which are self-adjoint $2^{\lfloor \frac{d}{2}\rfloor}\times 2^{\lfloor \frac{d}{2}\rfloor}$ complex matrices satisfying $\gamma_j\gamma_k+\gamma_k\gamma_j = 2\delta_{j,k}1$. The operator $D$ may be considered as an unbounded self-adjoint operator 
    on the Hilbert space $L_2(\T^d,\com^{2^{\lfloor \frac{d}{2}\rfloor}})$.
    The corresponding operator $F$ is a linear combination of Riesz transforms. The commutators of Riesz transforms and multiplication operators are studied in classical harmonic analysis: S. Janson
    and T. Wolff \cite{JW1982} proved that for $\qd f$ to be in $\mcL_p$ when $p > d$ it is necessary and sufficient that $f$ is in the Besov space $B^{\frac{d}{p}}_{p,p}(\T^d)$. On the other hand, Janson and Wolff also proved that if $p \leq d$ then
    $\qd f \in \mcL_{p}$ if and only if $f$ is a constant.
    
    A far more general characterisation of the spectral properties of commutators of Riesz transforms and multiplication operators was obtained by R. Rochberg and S. Semmes \cite{RS1989}.
    To date, investigations on the relationship between $f$ and $\qd f$ have been limited to the commutative case. To the best of our knowledge, the results treated in this paper are the 
    first concerning quantum differentiability in the strictly noncommutative setting.
    
    A related direction of research concerning quantised differentials is trace formulae. As early as \cite{Connes1988} it was known that for functions on compact manifolds, it is possible
    to express the Dixmier trace $\tr_\omega(|\qd f|^p)$ as an integral of a derivative of $f$ (See Subsection \ref{operator notation subsection} for the relevant definitions, and \cite[Chapter 6]{LSZ2012} for details
    on Dixmier traces).
    
    If $f \in C^\infty(\T^d)$, let $\nabla f = (\partial_1f,\partial_2f,\ldots,\partial_df)$ be the gradient vector of $f$, and let $\|\nabla f\|_2 = \left(\sum_{j=1}^d |\partial_j f|^2\right)^{\frac{1}{2}}$. Then as a special case of \cite[Theorem 3.3]{Connes1988} we have:
    \begin{equation}\label{torus trace formula}
        \tr_\omega(|\qd f|^d) = k_d\int_{\T^d} \|\nabla f(t)\|_{2}^ddm(t),
    \end{equation}
    where $k_d$ is a constant, and $m$ denotes the flat measure on $\T^d$ (i.e., the Haar measure). From the perspective of noncommutative geometry this formula ``shows how to pass from quantized $1$-forms to ordinary forms, not by a classical limit, but by a direct application
    of the Dixmier trace" \cite[Page 676]{Connes1988}. It is also possible to prove a similar formula for functions on the non-compact manifold $\real^d$, and indeed
    to extend the class of traces on the left hand side of \eqref{torus trace formula} to the much larger class of all continuous normalised traces on $\mcL_{1,\infty}$ \cite{LMSZ2017}.
         
    Recently there has been work on generalising the methods of harmonic analysis on tori to quantum tori.
    
    On a noncommutative torus $\qt$ (defined in terms of an arbitrary antisymmetric real $d\times d$ matrix $\theta$), it is possible to define analogues of many of the tools
    of harmonic analysis, such as differential operators and function spaces \cite{XXY2018} (see Section \ref{nc tori subsection}). In this setting, there are analogues of all of the components of \eqref{torus trace formula}, although the integral on the right
    must be replaced with the canonical trace associated to $\qt$. However the most straightforward generalisation of \eqref{torus trace formula}
    to $\qt$ is actually false. In this paper we state and prove a correct version of \eqref{torus trace formula} for noncommutative tori (Theorem \ref{trace formula}). The formula is stated for an appropriate class of elements
    $x \in {\hl L_2(\qt)}$ as:
    \begin{equation}\label{nc torus trace formula}
        \varphi(|\qd x|^d) = c_d\,\int_{\mathbb{S}^{d-1}} \tau\Bigg(\Big(\sum_{j=1}^d |\pd_j x-s_j\sum_{k=1}^d s_k\pd_k x|^2\Big)^{\frac{d}{2}}\Bigg)\,ds.
    \end{equation}
    Here, $\tau$ is the canonical trace associated to the noncommutative torus, and $c_d$ is a certain constant depending on $d$ (different to the constant $k_d$ in \eqref{torus trace formula}). The integral
    is over $s = (s_1,\ldots,s_d)$ in the $(d-1)$-dimensional sphere $\mathbb{S}^{d-1}$, with respect to its rotation-invariant measure $ds$.
    The partial derivatives $\{\partial_1x,\ldots,\partial_dx\}$ are defined in Subsection \ref{calculus definition subsubsection}.
    In the commutative case, the above formula reduces to \eqref{torus trace formula} (for a full comparison, see the discussion in Subsection \ref{commutative_discussion_subsection}). There are a number of nontrivial
    corollaries to \eqref{nc torus trace formula}, which we describe in the section below.

\subsection{Main results}
    We have three main results. We take $\theta$ to be an arbitrary $d\times d$ antisymmetric real matrix {\hl where $d \geq 2$}, in particular $\theta=0$ is not excluded. For further explanation of the notation, see Section \ref{notation section} below.
    
    Our first main result provides sufficient conditions for $\qd x \in \mcL_{d,\infty}$:
    \begin{thm}\label{sufficiency}
        If {\hl $x \in \dot{H}^1_d(\qt)$}, then {\hl $\qd x$ has bounded extension, and the extension is in $\mcL_{d,\infty}$}.
    \end{thm}
    {\hl The space $\dot{H}^{1}_d(\qt)$ is a noncommutative homogeneous Sobolev space }defined with respect to the partial derivatives $\partial_j$, $j=1,\ldots,d$ (these notions will be defined and discussed in Subsection \ref{calculus definition subsubsection}). We note that the above condition is similar to that in \cite[Theorem 11]{LMSZ2017}.    
    
    With Theorem \ref{sufficiency}, we can prove our second main result, the following trace formula:      
    \begin{thm}\label{trace formula}
    Let $x\in \dot H^1_d(\qt)$ be self-adjoint. Then there is a constant $c_d$ depending only on the dimension $d$ such that for any continuous normalised trace $\varphi$ on $\mcL_{1,\infty}$ we have:
        \begin{equation*}
            \varphi(|\qd x|^d) = c_d\,\int_{\mathbb{S}^{d-1}} \tau\Bigg(\Big(\sum_{j=1}^d |\pd_j x-s_j\sum_{k=1}^d s_k\pd_k x|^2\Big)^{\frac{d}{2}}\Bigg)\,ds.
        \end{equation*}
        Here, the integral over $\mathbb{S}^{d-1}$ is taken with respect to the rotation-invariant measure $ds$ on $\mathbb{S}^{d-1}$, and $s = (s_1,\ldots,s_d)$.
    \end{thm}
    As an aside we note that it is possible to give a short argument that the integrand above is continuous as a function of $s \in \mathbb{S}^{d-1}$.   
    
    Theorem \ref{trace formula}, in addition to being of interest in its own right, has a couple of corollaries, which to the best of our knowledge are novel.
    
    \begin{cor}\label{trace formula-bound}
    Let $x\in \dot H^1_d(\qt)$ be self-adjoint. Then there are constants $c_d$ and $C_d$ depending only on $d$ such that for any continuous normalised trace $\varphi$ on $\mcL_{1,\infty}$ we have 
    $$c_d \| x\|_{\dot{H}_d^1}^d \leq  \varphi(|\qd x|^d) \leq C_d \| x\|_{\dot{H}_d^1}^d .$$
    \end{cor}  
    
    As a converse to Theorem \ref{sufficiency}, we prove our third main result: the necessity of the condition {$x \in \dot H^1_d(\qt)$} for $\qd x \in \mcL_{d,\infty}$.
    \begin{thm}\label{necessity}
        Let {\hl$x \in L_2(\qt)$}. If $\qd x$ {\hl has bounded extension in $\mcL_{d,\infty}$ then $x \in \dot H^1_d(\qt)$.}
    \end{thm}
    {\hl The {\it a priori} assumption that $x \in L_2(\qt)$ can be justified as follows: $L_2(\qt)$ is the smallest class of $x$ where we can define $\qd x$ in a natural way. Furthermore, 
    one can motivate this assumption by noting that an $L_2$-condition is necessary and sufficient for Connes' trace theorem to hold in the commutative setting, see \cite[Theorem 2.5]{LPS2010} for details.}

    Since $\varphi$ vanishes on the trace class $\mcL_1$, Corollary \ref{trace formula-bound} immediately yields the following noncommutative version of the $p\leq d$ component of \cite[Theorem 1]{JW1982}:
    \begin{cor}\label{triviality}
        If $x \in {\hl L_2(\qt)}$ and $\qd x \in \mcL_{p}$, for $p \leq d$, then $x$ is a constant.
    \end{cor}
    Indeed, the $p \leq d$ component of \cite[Theorem 1]{JW1982} is an immediate and simple consequence of Corollary \ref{triviality} when $\theta = 0$.
    
    A further corollary of Theorem \ref{trace formula} is that $\varphi(|\qd x|^d)$ does not depend on the choice of continuous normalised trace $\varphi$. This
    implies certain asymptotic properties of the singular numbers of $\qd x$, beyond being merely in $\mcL_{d,\infty}$ \cite{KLPS,SSUZ2015}.
    
\subsection{Comparison to the commutative case}\label{commutative_discussion_subsection}
    Take $x \in \dot{H}^1_d(\qt)$. Consider the right hand side of the trace formula in Theorem \ref{trace formula},
    \begin{equation*}
        c_d\int_{\mathbb{S}^{d-1}} \tau\Bigg(\Big(\sum_{j=1}^d |\pd_j x-s_j\sum_{k=1}^d s_k\pd_k x|^2\Big)^{\frac{d}{2}}\Bigg)\,ds\,.
    \end{equation*}
    Define $\nabla x = (\partial_1 x,\partial_2x,\ldots,\partial_d x)$, and
    \begin{equation*}
        \|\nabla x\|_{2} := \Big(\sum_{j=1}^d |\pd_j x|^2\Big)^{\frac{1}{2}}.
    \end{equation*}
    In the commutative case (when $\theta = 0$), $x$ is a scalar valued function and $\|\nabla x\|_2^d$ coincides with the integrand in \eqref{torus trace formula}. Assuming commutativity,
    we can define the unit vector $u = \frac{\nabla x}{\|\nabla x\|_2}$ and take out a factor of $\|\nabla x\|_{2}^{\frac{d}{2}}$ to get:
    \begin{equation*}
        \tau\Bigg(\int_{\mathbb{S}^{d-1}} \Big(\sum_{j=1}^d |\pd_j x-s_j\sum_{k=1}^d s_k\pd_k x|^2\Big)^{\frac{d}{2}}\,ds\Bigg) = \tau\Bigg(\|\nabla x\|_2^{\frac{d}{2}}\int_{\mathbb{S}^{d-1}} \Big(\sum_{j=1}^d|u_j-s_j\sum_{k=1}^d s_k u_k |^2\Big)^{\frac{d}{2}}\,ds  \Bigg)
    \end{equation*}
    (where $u = (u_1,u_2,\ldots,u_d)$, and the interchange of $\tau$ and the integral is easily justified by Fubini's theorem in the commutative case).
    However, since the measure $ds$ on $\mathbb{S}^{d-1}$ is invariant under rotations, we can choose coordinates $\{e_1,\ldots,e_d\}$ for $\real^d$ so that $u = e_1$, and then:
    \begin{equation*}
        b_d := \int_{\mathbb{S}^{d-1}} \Big(\sum_{j=1}^d|u_j-s_j\sum_{k=1}^d s_k  u_k  |^2\Big)^{\frac{d}{2}}\,ds
    \end{equation*}
    is independent of $u$, and is a constant scalar. Thus in the commutative case we have:
    \begin{equation*}
        \varphi(|\qd x|^d) = c_db_d\tau(\|\nabla x\|_2^d).
    \end{equation*}
    This recovers \eqref{torus trace formula} upon taking $k_d = c_db_d$.
    In the noncommutative case, we cannot take out a factor of $\|\nabla x\|_2^2$, and this explains why the form of the right hand side
    of Theorem \ref{trace formula} is more complicated than $k_d\tau(\|\nabla x\|_2^d)$.

\section{Notation}\label{notation section}

\subsection{Operators, Ideals and traces}\label{operator notation subsection}
    The following material concerning operator ideals and traces is standard. For more details we refer the reader to \cite{LSZ2012, Simon1979}.
    Let $H$ be a complex separable Hilbert space, and let $\Bc(H)$ denote the set of bounded operators on $H$, and let $\cK(H)$ denote the ideal of compact operators on $H$. Given $T\in \cK(H)$, the sequence of singular values $\mu(T) = \{\mu(k,T)\}_{k=0}^\infty$ is defined as:
    \begin{equation*}
        \mu(k,T) = \inf\{\|T-R\|\;:\;\rank(R) \leq k\}.
    \end{equation*}
    Equivalently, $\mu(T)$ is the sequence of eigenvalues of $|T|$ arranged in non-increasing order with multiplicities.
    
    Let $p \in (0,\infty).$ The Schatten class $\mcL_p$ is the set of operators $T$ in $\cK(H)$ such that $\mu(T)$ is $p$-summable, i.e. in the sequence space $\ell_p$. If $p \geq 1$ then the $\mcL_p$
    norm is defined as:
    \begin{equation*}
        \|T\|_p := \|\mu(T)\|_{\ell_p} = \left(\sum_{k=0}^\infty \mu(k,T)^p\right)^{1/p}.
    \end{equation*}
    With this norm $\mcL_p$ is a Banach space, and an ideal of $\Bc(H)$.
    
    Analogously, the weak Schatten class $\mcL_{p,\infty}$ is the set of operators $T$ such that $\mu(T)$ is in the weak $L_p$-space $\ell_{p,\infty}$, with quasi-norm:
    \begin{equation*}
        \|T\|_{p,\infty} = \sup_{k\geq 0} (k+1)^{1/p}\mu(k,T) < \infty.
    \end{equation*}
    As with the $\mcL_p$ spaces, $\mcL_{p,\infty}$ is an ideal of $\Bc(H)$. We also have the following form
    of H\"older's inequality,
    \begin{equation}\label{weak-type-holder}
        \|TS\|_{r,\infty} \leq c_{p,q}\|T\|_{p,\infty}\|S\|_{q,\infty}
    \end{equation}
    where $\frac{1}{r}=\frac{1}{p}+\frac{1}{q}$, for some constant $c_{p,q}$.
    
    Of particular interest is $\mcL_{1,\infty}$, and we are concerned with traces on this ideal. For more details, see \cite[Section 5.7]{LSZ2012} and \cite{SSUZ2015}. A functional $\varphi:\mcL_{1,\infty}\to \Cplx$ is called a trace if it is unitarily invariant. That is, for all unitary operators
    $U$ and $T\in \mcL_{1,\infty}$ we have that $\varphi(U^*TU) = \varphi(T)$. It can then be shown that for all bounded operators $B$ we have $\varphi(BT)=\varphi(TB).$
    
    An important fact about traces is that any trace $\varphi$ on $\mcL_{1,\infty}$ vanishes on $\mcL_1$ \cite[Theorem~5.7.8]{LSZ2012}. A trace $\varphi$ is called continuous if it is continuous with respect to the $\mcL_{1,\infty}$ quasi-norm. It is known that not all traces on $\mcL_{1,\infty}$ are continuous \cite[Remark~3.1(3)]{LSZ2018}. 
    Within the class of continuous traces on $\mcL_{1,\infty}$ there are the well-known Dixmier traces \cite[Chapter 6]{LSZ2012}.
    
    Finally, we say that a trace $\varphi$ on $\mcL_{1,\infty}$ is normalised if $\varphi$ takes the value $1$ on any compact positive operator with eigenvalue sequence $\{\frac{1}{n+1}\}_{n=0}^\infty$ (any two such
    operators are unitarily equivalent, and so the particular choice of operator is inessential).

\subsection{Noncommutative Tori}\label{nc tori subsection}
    Harmonic analysis on noncommutative tori is an established subject. The exposition here closely follows \cite{XXY2018}, and for sake of brevity we refer the reader to \cite{XXY2018}
    for a detailed exposition of the topic and provide here only the definitions relevant to this text.
\subsubsection{Basic definitions}
    We fix an integer $d > 1$ and $\theta = \{\theta_{j,k}\}_{j,k = 1}^d$, a $d\times d$ antisymmetric real matrix. The $C^*$-algebra of continuous functions on the noncommutative torus, denoted $C(\qt)$, is the universal $C^*$-algebra on $d$ unitary generators $U_1,\ldots,U_d$
    which satisfy:
    \begin{equation*}
        U_jU_k = e^{2\pi \ri \theta_{j,k}}U_kU_j,\quad 1\leq j,k\leq d.
    \end{equation*}
    Given $n = (n_1,\ldots,n_d) \in \ent^d$, we adopt the shorthand notation:
    \begin{equation*}
        U^n := U_1^{n_1}U_2^{n_2}\cdots U_d^{n_d}.
    \end{equation*}
    There exists an action $\alpha$ of the torus group $\T^d$ on $C(\qt)$, given on a generator $U_j$ by:
    \begin{equation}\label{action-T-qt}
        \alpha_z(U_j) = z_jU_j, \quad z = (z_1,z_2,\ldots, z_d)\in \T^d.
    \end{equation}
    The action $\alpha$ can be extended to a norm-continuous group of automorphisms of $C(\qt)$. There is a distinguished trace state $\tau$
    on $C(\qt)$, which may be constructed in several ways, one of which is by averaging over $\alpha$ as follows:
    It can be shown that the fixed point subalgebra of $C(\qt)$ under the action of $\alpha$ is exactly the trivial subalgebra $\Cplx 1$. Hence if $x \in C(\qt)$ then
    averaging over $\T^d$ with respect to the Haar measure $m$ on $\T^d$:
    \begin{equation*}
        \int_{\T^d} \alpha_z(x)\,dm(z)
    \end{equation*}
    yields a multiple of the identity element. Defining
    \begin{equation*}
        \tau(x)1 = \int_{\T^d} \alpha_z(x)\,dm(z)
    \end{equation*}
    yields the canonical trace state $\tau$ on $C(\qt)$. Given $\tau$ we can now define the GNS Hilbert space $L_2(C(\qt),\tau)$, which we denote $L_2(\qt)$, and we identify $C(\qt)$ as an algebra of bounded
    operators on $L_2(\qt)$, where $x \in C(\qt)$ acts on $\xi \in L_2(\qt)$ by left multiplication. Taking the weak operator topology closure $C(\qt)''$ in $\Bc(L_2(\qt))$ yields a von Neumann algebra, which we denote $L_\infty(\qt)$. 
    
    The $L_p$-spaces for $p \in [1,\infty)$ on $\qt$ are then defined as the operator $L_p$-spaces \cite{PX2003,LSZ2012} on $(L_\infty(\qt),\tau)$,
    \begin{equation*}
        L_p(\qt) := \mcL_p(L_\infty(\qt),\tau).
    \end{equation*} 
    
    For $x \in L_1(\qt)$ and $n\in \ent^d$, we define:
    \begin{equation*}
        \wh{x}(n) = \tau(x(U^{n})^{*}).
    \end{equation*}
    
    By the definition of $\tau$, we see that $\tau(U_n) = \delta_{n,0}$, and then standard Hilbert space arguments show that any $x \in L_2(\qt)$ can be written as an $L_2$-convergent
    series:
    \begin{equation*}
        x = \sum_{n \in \ent^d} \wh{x}(n)U^n,
    \end{equation*}
 with 
\beq\label{Plancherel-qt}
\|x\|_2^2 = \sum_{n\in \ent^d} | \wh x(n) |^2 .
\eeq

    The space $\Sc(\qt)$ is defined to be the subset of $x\in C(\qt)$ such that the sequence of Fourier coefficients $\{\wh{x}(n)\}_{n\in \ent^d}$
    has rapid decay (i.e., the sequence $\{|\wh{x}(n)|\}_{n\in \ent^d}$ is eventually dominated by the reciprocal of any polynomial). We may consider $\Sc(\qt)$
    as the space of smooth functions on $\qt$, since in the commutative setting this space corresponds with the space of $C^\infty$ functions. There is also a canonical Fr\'echet 
    topology on $\Sc(\qt)$, and the space {\hl $\Dist'(\qt)$}, called the space of distributions on $\qt$, is defined to be the topological dual of $\Sc(\qt)$.

\subsubsection{Calculus for quantum tori}\label{calculus definition subsubsection}
    Many aspects of harmonic analysis on $\T^d$ carry over to $\qt$. For example we may define the partial differentiation operators $\pd_j$, $j = 1,\cdots,d$ by:
    \begin{equation*}
        \pd_j (U^n) = 2\pi \ri n_jU^n,\quad n = (n_1,\ldots,n_d) \in \ent^d.
    \end{equation*}
    Every partial derivation $\pd_j$ can be viewed a densely defined closed (unbounded)
    operator on $L_2(\qt)$, whose adjoint is equal to $-\pd_j$.  Let $\Delta=\partial_1^2+\cdots+\partial_d^2$ be the Laplacian. Then $\Delta = - (\pd_1^* \pd_1 + \cdots +\pd_d^* \pd_d)$, so $-\Delta$ is a positive operator on $L_2(\qt)$ with spectrum equal to $\{4\pi^2 |n|^2  : n \in \ent^d\}$. As in the Euclidean case, we let $D_j  = -\ri \pd_j$, which is then self-adjoint.
    Given $n =(n_1,\cdots,n_d)\in \mathbb{N}_0^d$ ($\mathbb{N}_0$ denoting the set of nonnegative integers), the associated partial derivation $D^n$ is defined to be
    $D_1^{n_1}\cdots D_d^{n_d}$. The order of $D^n$ is   $|n|_1=n_1+\cdots+ n_d$. 
    By duality, the derivations transfer to  $\Dist'(\qt)$  as well.

    For $\al \in \real$, denote by $J^\al$ the $\al$-order Bessel potential $(1 -   \Delta )^{\frac{\al}{2}}$.
    The potential (or fractional)  Sobolev space of order $\al \in \real$  is defined to be
    \begin{equation}\label{bessel-sobolev-space-definition}
        H_p^\alpha(\mathbb{T}_{\theta}^d)=\big\{ x\in\Dist'(\qt) : J^\al x\in L_p(\mathbb{T}_{\theta}^d) \big\},
    \end{equation}
    equipped with the  norm
    $$\|x\|_{H_p^\alpha}=\|J^\al x\|_p\,.$$ Since $J^0$ is the identity, $H_p^0(\qt) = L_p(\qt)$.
    {\hl 
    As in the classical case, if $\alpha$ is a non-negative integer then $H^{\alpha}_p(\qt)$ admits an equivalent norm in terms of the sum of the $p$-norms of the partial derivatives of order up to $\alpha$.
    To be explicit, the Sobolev space of order $k\in \mathbb{N}$  on $\mathbb{T}_{\theta}^d$ may be described as:
    $$H_p^k(\mathbb{T}_{\theta}^d)= \big\{ x\in\Dist'(\qt) :  D^n x \in L_p(\mathbb{T}_{\theta}^d) \textrm{ for each }n\in \mathbb {N}_0^d \textrm{ with } |n |_1\leq k \big\},$$
    equipped with the  norm
    $$\|x\|_{H_p^k}=\Big(\sum_{0\leq |n |_1\leq k}\|D^n x\|_{p}^p\Big)^{\frac{1}{p}}.$$
    
        The equivalence of the above norm and the Bessel potential norm $\|J^{k} x\|_p$ is a well-established fact in the theory of harmonic analysis on $\qt$, being proved in the $p=2$ case by \cite[Theorem 2.1]{Spera1992}
        and a later proof for general $p$ can be found as \cite[Theorem 2.9]{XXY2018}.
    }

   In this paper, we will mainly use the ``homogeneous" Sobolev space $\dot{H}^1_p (\qt)$ and the potential Sobolev spaces $H_2^\alpha(\mathbb{T}_{\theta}^d)$. The norm of $\dot{H}^1_p (\qt)$ with $ p \geq 2$, may be described in the following equivalent forms:
    \beq\label{Sob-equi-norm}
    \|x\|_{\dot{H}^1_p}
    = \Big(\sum_{j=1}^d \|\pd_j x\|_p^p \Big)^{\frac{1}{p}}\approx \sum_{j=1}^d \|\pd_j x\|_p \approx \| (\sum_{j=1}^d    |\pd_j x |^2 ) ^{\frac 1 2 } \|_p ,
    \eeq
    where the relevant constants depend only on $d$ and $p$. Then $\dot{H}^{1}_p(\qt)$ may be defined as the subspace of $\Dist'(\qt)$ for which the above norm is finite. Note that the difference between $\dot{H}^{1}_p(\qt)$ and $H^1_p(\qt)$
    is that for $\dot{H}^{1}_p(\qt)$ we do not assume that the $L_p$-norm is finite. 
    For any $x\in H^1_p (\qt)$, we have the following Poincar\'e type inequality
    \beq\label{Poincare}
    \|x- \wh x (0) \|_p \leq  C_{p,d} \|x\|_{\dot{H}^1_p} .
    \eeq
    See \cite[Theorem~2.12]{XXY2018}. For every $\al \in \real$, the space $H_2^\al(\T_\theta^d)$ is a Hilbert space with the inner product
    $$\lan x , y \ran = \tau ( J^\al y^*  J^\al  x )  .$$
    {\hl 
    It is proved in \cite[Theorem 3.3]{Spera1992} and \cite[Proposition 9.2]{HLP2018a} for arbitrary real $\al, \beta \in \mathbb{R}$ with $\al > \beta$, the embedding
        \beq\label{Sob-cpt-emb}
            H_2^\al(\T_\theta^d)    \hookrightarrow H_2^\beta (\T_\theta^d) \quad \mbox{is compact}.
        \eeq
    }

  The Dirac operator $\D$ {\hl (more precisely, the spin-Dirac operator)} is defined in terms of $\gamma$ matrices in direct analogy to commutative tori. Define $N = 2^{\lfloor \frac{d}{2}\rfloor}$ and select $N\times N$ complex self-adjoint matrices
    $\{\gamma_1,\ldots,\gamma_d\}$ satisfying $\gamma_j\gamma_k +\gamma_k\gamma_j = 2\delta_{j,k}1$, and define:
    \begin{equation*}
        \D = \sum_{j=1}^d \gamma_j\otimes D_j
    \end{equation*}
    as an unbounded, densely defined linear operator on the Hilbert space $\Cplx^N\otimes L_2(\qt)$. {\hl This definition coincides with \cite[Definition 12.14]{green-book}. }
    
    Since all $D_j$'s
are self-adjoint, $\D$ is also self-adjoint. We then define the sign of $\D$ via the Borel functional calculus, which can be expressed as 
$$\sgn (\D)    = \sum_{j=1}^d \g_j \ot  \frac{D_j }{\sqrt{ D_1^2 + D_2^2 +\cdots + D_d^2}} \,.$$

Given $x\in L_\8(\qt)$, denote by $M_x: y \mapsto xy$ the operator of left multiplication on $L_2(\qt)$. The operator $1\ot M_x$ is a bounded linear operator on $\com^N \ot  L_2(\qt) $, where $1$ denotes the identity operator on $\com^N$. The commutator
$$\qd x : = \ri [\sgn(\D) , 1\ot M_x],\quad x \in L_\infty(\qt)$$
denotes the quantised differential on quantum tori.

{\hl 
    On the other hand, if $x$ is not necessarily bounded we may still define $\qd x$ on the dense subspace $C^\infty(\qt)\otimes \com^N$ as follows. Suppose that $x \in L_2(\qt)$. Then if $\eta \in C^\infty(\qt)\otimes \com^N$, we will
    have $(1\otimes M_x)\eta \in L_2(\qt)\otimes \com^N$. Moreover, $\sgn(D)\eta$ is still in $C^{\infty}(\qt)\otimes \com^N$ since by definition an element of $C^{\infty}(\qt)$ has Fourier coefficients of rapid decay, and $\sgn(D)$ is represented as a
    Fourier multiplier with bounded symbol.
    Thus the expression:
    \begin{equation*}
        (\qd x)\eta := \ri\sgn(D)(1\otimes M_x)\eta - \ri(1\otimes M_x)\sgn(D)\eta
    \end{equation*}
    is a well-defined element of $L_2(\qt)\otimes \com^N$ for all $\eta \in C^{\infty}(\qt)\otimes \com^N$. 
}

\subsubsection{Fourier multipliers for quantum tori}
    Let $g$ be a bounded scalar function on $\ent^d$. For $x \in L_2(\qt)$, the Fourier multiplier $T_g$ with symbol $g$ is defined
    on $x$ by:
    \begin{equation}\label{def-Fourier-Z}
        T_gx = \sum_{n \in \ent^d} g(n)\wh{x}(n)U^n.
    \end{equation}
    By virtue of the Plancherel identity \eqref{Plancherel-qt}, $T_g$ indeed defines a bounded linear operator on $L_2(\qt)$ and the above series converges in the $L_2$-sense.
    If $g$ is unbounded, we may define $T_g$ on the dense subspace of $L_2(\qt)$ of those $x$ with finitely many non-zero Fourier coefficients.
    
    An equivalent perspective on Fourier series is to consider a function $\phi \in L_1(\T^d)$ on the commutative torus. We may then define
    the convolution of $\phi$ with $x \in L_2(\qt)$ by:
    \begin{equation*}
        \phi\ast x =\int_{\T^d}\al_w(x)\phi(w) dw.
    \end{equation*}
    
    In terms of Fourier coefficients, we have:
    \begin{equation*}
        \phi\ast x = T_{\widehat{\phi}}x.
    \end{equation*}
    
    Fourier multipliers for quantum tori were studied in detail in \cite[Chapter 7]{XXY2018} (there, $T_{g}$ was denoted by $M_g$).
    From the perspective of functional calculus, we may also write:
    \begin{equation*}
        T_g = g(\frac{1}{2\pi \ri}\partial_1, \frac{1}{2\pi \ri}\partial_2,\ldots,\frac{1}{2\pi \ri}\partial_d).
    \end{equation*}

 The above defined derivatives $D^\al$, Laplacian $\Delta$, and Bessel potential $J^\al$ may all be viewed as Fourier multipliers: the symbol of $D_j$ is $2\pi \xi_j$; the symbol of $\Delta$ is $-|2\pi \xi |^2$; and the symbol of $J^\al$ is $(1+ |  2\pi \xi  |^2) ^{\frac{\al}{2}}$. We will denote by $\Jx$ the function $(1+ |   \xi  |^2) ^{\frac{1}{2}}$ in the sequel.
 
    A far reaching extension of the notion of a Fourier multiplier is a pseudodifferential operator. We outline the pseudodifferential operator theory for the noncommutative torus in Section \ref{sec-pdo}.

\section{Cwikel-type estimates for quantum tori}
    In the classical, commutative setting, Cwikel estimates are bounds on the singular values of operators of the form:
    \begin{equation*}
        M_fg(-\ri\nabla)
    \end{equation*}
    where $f$ and $g$ are essentially bounded functions on $\real^d$, and $M_f$ and $g(-\ri\nabla)$ denote pointwise multiplication and Fourier multiplication 
    on $L_2(\real^d)$ respectively (see e.g. \cite[Chapter 4]{Simon1979} and \cite{Cwikel1977}).
    
    In the setting of noncommutative tori, we instead consider operators of the form $M_xT_g$, where $x\in L_{\infty}(\qt)$ and $g \in \ell_{\infty}(\ent^d)$. We can obtain 
    the following as a special case of \cite{LeSZ2017}:
    
    \begin{thm}\label{Cwikel-type}
        \begin{enumerate}[\rm (i)]
            \item\label{L_p cwikel} If $x\in L_p(\qt)$ and $g\in \ell_p(\ent^d)$ with $2\leq p <\8$, then $M_x \,T_g$ is in $\mcL_p$ and 
                $$\|M_x\, T_g\|_{\mcL_p} \leq C_p \|x\|_p \|g\|_p.$$
            \item\label{weak L_p cwikel} If $x\in L_p(\qt)$ and $g\in \ell_{p,\8}(\ent^d)$ with $2< p <\8$, then $M_x \,T_g$ is in $\mcL_{p,\8} $ and 
                $$\|M_x\, T_g\|_{\mcL_{p,\8}} \leq C_p \|x\|_p \|g\|_{p,\8}.$$
        \end{enumerate}
    \end{thm}
    \begin{proof}
        We in fact prove the following far stronger estimate, stated in the language of symmetric function spaces \cite[Chapter 2]{LSZ2012}: For any symmetric function space $E$ whose norm satisfies the Fatou property\footnote{meaning that if $A_n$ is a sequence of positive operators with $A_n\uparrow A$ in the weak operator topology, then $\|A\|_{E} \leq \sup_{n} \|A_n\|_{E}$} and is an interpolation space of $L_2$ and $L_\infty$, if $x\otimes g \in E(L_{\infty}(\qt)\otimes \ell_{\infty}(\ent^d))$ then $M_xT_g$ is in $E(\Bc(L_2(\qt)))$, with norm bound,
        \begin{equation}\label{interpolated cwikel estimate}
            \|M_xT_g\|_{E(\Bc(L_2(\qt)))} \leq C_E \|x\otimes g\|_{E(L_\infty(\qt)\otimes \ell_\infty(\ent^d))}.
        \end{equation}
        
        After proving \eqref{interpolated cwikel estimate}, we explain how it entails the results in the statement of the theorem.
        
        In fact \eqref{interpolated cwikel estimate} can be obtained by a direct application of \cite[Corollary 3.5]{LeSZ2017}. Here we have two von Neumann algebras $L_{\infty}(\qt)$
        and $\ell_{\infty}(\ent^d)$ represented on the same Hilbert space $L_2(\qt)$ by left multiplication and Fourier multiplication respectively. In this setting, we can use \cite[Corollary 3.5]{LeSZ2017} which states that if we have an estimate of the form:
        \begin{equation}\label{L_2 cwikel}
            \|M_xT_g\|_{\mcL_{2}(\Bc(L_{2}(\qt)))} \leq \|x\|_{L_2(\qt)}\|g\|_{\ell_2(\ent^d)}
        \end{equation}
        then \eqref{interpolated cwikel estimate} follows.
        
        To prove \eqref{L_2 cwikel}, we can express the Hilbert-Schmidt norm in terms of an expansion with respect to the basis $\{U^m\}_{ m \in \ent^d}$ of $L_2(\qt)$,
        \be\begin{split}
        \| M_x T_g \|_{\mcL_2}^2 & = \sum_{m, n \in \ent^d}  \Big|    \tau \Big(x  \big(T_gU^m\big) (U^n)^* \Big) \Big|^2\\
        & = \sum_{m, n \in \ent^d}  \Big|    \tau \Big(x   g(m )U^m  (U^n)^* \Big) \Big|^2 = \sum_{m, n \in \ent^d} |g(m)|^2 \Big|    \tau \Big(x  U^m  (U^n)^* \Big) \Big|^2 \\
        &= \sum_{m \in \ent^d} |g(m)|^2 \sum_{n \in \ent^d} \Big|    \tau \Big(x  U^m  (U^n)^* \Big) \Big|^2 .
        \end{split}\ee
        By the Plancherel formula \eqref{Plancherel-qt}, we have
        $$\sum_{n \in \ent^d}  |    \tau  (x  U^m  (U^n)^*  )  |^2 = \|x U^m \|_2^2= \|x\|_2^2.$$
        Thus,
        $$
        \| M_x T_g \|_{\mcL_2}^2 =  \|x\|_2^2 \|g\|_2^2.$$
        Hence, \eqref{L_2 cwikel} holds and thus by \cite[Corollory 3.5]{LeSZ2017} it follows that \eqref{interpolated cwikel estimate} holds.
        
        Now, we take $E = L_p$ in \eqref{interpolated cwikel estimate} for $p \in (2,\infty)$. This is indeed an interpolation space between $L_2$ and $L_\infty$ whose norm satisfies the Fatou property. 
         Then combining \eqref{interpolated cwikel estimate} with the identity
        \begin{equation*}
            \|x\otimes g\|_{L_p(L_\infty(\qt)\otimes \ell_\infty(\ent^d))} = \|x\|_p \|g\|_p
        \end{equation*}
        yields \eqref{L_p cwikel}.
        
        Finally, to obtain \eqref{weak L_p cwikel}, we take $E=L_{p,\infty}$ in \eqref{interpolated cwikel estimate} and use the estimate:
        \begin{equation*}
            \|x\otimes g\|_{L_{p,\infty}(L_\infty(\qt)\otimes \ell_\infty(\ent^d))}\leq \|x\|_p \|g\|_{p,\infty}.
        \end{equation*}
        This completes the proof of \eqref{weak L_p cwikel}.
    \end{proof}

        Consider the function on $\ent^d$, $n\mapsto (1+|n|^2)^{-\frac{d}{2}}$. When $|n| > 1 $, we have $(1+|n| ^2)^{-\frac{d}{2}} \leq |n|^{-d}$. For $|n| \leq 1$, $(1+|n| ^2)^{-\frac{d}{2}}$ is bounded from above by $1$. Hence $n \mapsto (1+|n| ^2)^{-\frac{d}{2}} \in \ell_{1,\infty} (\ent^d)$, and so $n\mapsto (1+|n| ^2)^{-\frac{\bt}{2}}\in \ell_{\frac{d}{\bt} , \8} (\ent^d)$. Then it follows immediately from the above theorem that

        \begin{cor}\label{Cwikel-type-cor}
        Consider the linear operator $(1\ot x) (1+\D^2)^{-\frac \beta  2} $ on $\com^N \ot L_2(\qt)$. If $x\in L_{\frac d \beta}(\qt)$ with $\frac d \beta >2$, then $(1\ot M_x) (1+\D^2)^{-\frac \beta  2} \in \mcL_{\frac d  \beta, \8} ,$ and
        $$\|(1\ot M_x) (1+\D^2)^{-\frac \beta 2}\|_{\mcL_{\frac d \beta, \8}} \leq  C\|x\|_{\frac d \beta},$$
        where the constant $C>0$ depends only on $d$ and $\beta$.
        \end{cor}
        
    At this point it is worth noting that since the function $n\mapsto (1+|n|^2)^{-\frac{\al}{2}}$ is in $\ell_{\frac{d}{\al},\8}(\ent)$, for all $\al > 0$ we have:
    \begin{equation}\label{bessel potential ideal}
        {\hl J^{-\alpha} \in \mcL_{\frac{d}{\al},\infty}.}
    \end{equation}

\section{Proof of Theorem \ref{sufficiency}}\label{sec-suff}

This section is devoted to the proof of Theorem \ref{sufficiency}, that is, that the condition $x \in {\hl \dot{H}^{1}_d(\qt)}$ is sufficient
for $\qd x \in \mcL_{d,\infty}$, and with an explicit norm bound: {\hl $$\|\qd x\|_{d,\infty} \leq C_d\|x\|_{\dot{H}^{1}_d(\qt)}.$$
Note that due to the Poincar\'e inequality \eqref{Poincare}, $\dot{H}^1_d(\qt)$ is a subset of $L_d(\qt)$, and thus in particular the operator $\qd x$ is well-defined.}

The following lemma is a corollary of Theorem \ref{Cwikel-type}\eqref{L_p cwikel}.

\begin{lem}\label{Cwikel-xp}
Suppose that $p>\frac d 2$ and $x\in L_p(\qt)$. If $p\geq 2$, then there exists a constant $C_{p,d}>0$ such that 
$$ \big\|\big[\sgn(\D)  - \frac{\D}{\sqrt{1+\D^2}}, 1\ot M_x  \big]\big\|_{\mcL_p}  \leq C_{p,d}  \|x\|_p ,        $$
{\hl meaning that, if $x \in L_p(\qt)$ then the above commutator (initially defined on $C^\infty(\qt)\otimes \com^N$) admits an extension to a bounded operator which is in the ideal $\mcL_p$ with the above norm bound.}
\end{lem}

\begin{proof}
Let $1\leq j \leq d$, and for $n \in \ent^d$ define
$$h_j(n) := \frac{n_j}{|n|} -\frac{n_j}{((2\pi)^{-2}+|n|^2)^{\frac{1}{2}}}.$$
Thus,
\begin{equation*}
    T_{h_j} = h_j(-\frac{\ri}{2\pi}\nabla) = \frac{-\ri\partial_j}{\sqrt{-\Delta}}-\frac{-\ri\partial_j}{(1-\Delta)^{\frac{1}{2}}}
\end{equation*}
and so,
\begin{align*}
    \sgn(\D) - \frac{\D}{\sqrt{1+\D^2}} &= \sum_{j=1}^d \gamma_j\otimes \left( \frac{-\ri\partial_j}{\sqrt{-\Delta}}-\frac{-\ri\partial_j}{(1-\Delta)^{\frac{1}{2}}}\right)\\
                                        &= \sum_{j=1}^d \gamma_j\otimes h_j(-\frac{\ri}{2\pi}\nabla)\\
                                        &= \sum_{j=1}^d \gamma_j\otimes T_{h_j}.
\end{align*}
One can easily check that $h_j \in \ell_p(\ent^d)$ as $p> \frac d 2$. Expanding out the commutator,
\be 
\begin{split}
\big[\sgn(\D)  - \frac{\D}{\sqrt{1+\D^2}}, 1\ot M_x  \big] &= \big[\sum_{j=1}^d \g _j  \ot T_{h_j}, 1\ot M_x  \big]\\
&=\sum_{j=1}^d  \g_j \ot  [T_{h_j}, M_x].
\end{split}
\ee
Hence, 
\be 
\begin{split}
\|\big[\sgn(\D)  - \frac{\D}{\sqrt{1+\D^2}}, 1\ot M_x  \big]\|_{\mcL_p} &\leq d \max_{1\leq j \leq d} \|\big[ T_{h_j},  M_x  \big]\|_{\mcL_p}\\
&\leq d \max_{1\leq j \leq d} \big(\| T_{h_j} M_x   \|_{\mcL_p}+\| M_x   T_{h_j}\|_{\mcL_p}\big)\\
&= d \max_{1\leq j \leq d} \big(\| M_{x ^*}  T_{h_j}\|_{\mcL_p}+   \|M_x  T_{h_j}\|_{\mcL_p}\big).
\end{split}
\ee
The desired conclusion follows then from Theorem \ref{Cwikel-type}.(i).
\end{proof}

The proof of the next lemma relies on the technique of double operator integrals (see \cite{PSW2002} and \cite{PS2009} and references therein). 
Let $H$ be a (complex) separable Hilbert space. Let $D_0$ and $D_1$ be self-adjoint (potentially unbounded) operators on $H$, and $E^0$ and $E^1$ be the associated spectral measures. For all $x, y \in \mcL_2(H)$, the measure $(\lambda, \mu)  \mapsto \rTr(x\, dE^0(\lambda) \, y \, dE^1(\mu)  )$ is a countably additive complex valued measure on $\real^2$. We say that $\phi \in L_\8(\real^2)$ is $E^0 \otimes E^1$ integrable if there exists an operator $T_\phi ^{D_0, D_1} \in \mathcal{B}  (\mcL_2(H))$ such that for all $x, y \in \mcL_2(H)$,
$$\rTr (x\,T_\phi ^{D_0, D_1} y  ) =\int _{\real^2}    \phi(\lambda, \mu )   \rTr(x\, dE^0(\lambda) \, y \, dE^1(\mu)  ). $$
The operator $T_\phi ^{D_0, D_1} $ is called the transformer. For $A\in \mcL_2(H) $, we define
\beq\label{doi-def}
T_\phi ^{D_0, D_1}(A)=\int _{\real^2}    \phi(\lambda, \mu )   dE^0(\lambda) \, A \, dE^1(\mu)  . \eeq
This is called a double operator integral.

\begin{lem}\label{commutator-Sob}
Let $x\in \dot H_d^1(\qt)$. Then 
$$ \big\|\big[\frac{\D}{\sqrt{1+\D^2}}, 1\ot M_x  \big]\big\|_{\mcL_{d,\8} }  \leq B_{d}  \|x\|_{\dot H_d^1}$$
where the constant $B_d>0$ depends only on $d$.
{\hl As with Lemma \ref{Cwikel-xp}, the above commutator is interpreted as being initially defined on $C^\infty(\qt)\otimes \com^N$.}
\end{lem}

\begin{proof}
Set $g(t)= t (1+t^2)^{-\frac{1}{2}}$ for $t\in \real$. {\hl Suppose initially that $x \in C^{\infty}(\qt)$. Under this assumption, $[D,1\otimes M_x]$ extends to a bounded operator, and thus we can apply \cite[Theorem~4.1]{BS1989} (see also Proposition 2.6 and Theorem 3.1 in \cite{PS2008}) to get}
\beq\label{repr-commutator}
[g(\D), 1\ot M_x] =  T_{g^{[1]}}^{\D,\D} ([\D, 1\ot M_x]),
\eeq
where $g^{[1]}(\lambda,\mu )  := \frac{g(\lambda)-g(\mu )}{\lambda-\mu }$ for different $\lambda, \mu \in \real$. By  \cite[Lemma 9]{LMSZ2017}, we have $g^{[1]} =\p_1\p_2 \p_3$, with 
$$\psi_1 = 1+ \frac{1-\lambda \mu }{(1+ \lambda^2)^{\frac 1 2 } (1+\mu ^2) ^{\frac 1 2 } },\;\; \psi_2 = \frac{(1+\lambda^2)^{\frac 1 4} (1+\mu ^2)^{\frac 1 4} }{(1+ \lambda^2)^{\frac 1 2 } + (1+\mu ^2) ^{\frac 1 2 } },\;\;\psi_3 =  \frac{1   }{(1+ \lambda^2)^{\frac 1 4 } (1+\mu ^2) ^{\frac 1 4 } }.$$ 
It follows that 
\beq\label{TDDg}
T_{g^{[1]}}^{\D,\D}  = T_{\p_1}^{\D,\D}  T_{\p_2}^{\D,\D}  T_{\p_3}^{\D,\D} .
\eeq
By \cite[Lemma 8]{LMSZ2017}, we see that the transformer $T_{\p_2}^{\D,\D} $ is bounded on both $\mcL_1$ and $\mcL_\8$. 

{\hl For $k=1,3$ the function $\psi_k$ can be written as a linear combination of products of bounded functions of $\lambda$ and of $\mu$, and from this it follows that $T_{\p_k}^{\D,\D}$ is a bounded linear map on $\mcL_{1}$
and $\mcL_\8$. For further details, see e.g. \cite[Corollary 2]{PS2009} and \cite[Corollary 2.4]{RX2011}.}

Then by real interpolation of $(\mcL_1, \mcL_\8)$ (see \cite{DDP1992} or \cite{Xu2007}), the transformers $T_{\p_k}^{\D,\D} $ with $k=1,2,3$ are bounded linear transformations from $\mcL_{d,\8}$ to $\mcL_{d,\8}$.

We now exploit the identity in \eqref{repr-commutator} and the product of terms in \eqref{TDDg}, noticing that 
\be\begin{split}
\|[g(\D), 1\ot M_x]\|_{\mcL_{d,\8}}&\leq \| T_{\p_1}^{\D,\D} \|_{\mcL_{d,\8}\ra \mcL_{d,\8}}  \| T_{\p_2}^{\D,\D} \|_{\mcL_{d,\8}\ra \mcL_{d,\8}}\\
&\;\;\;\;\;\;\;\;    \times  \| T_{\p_3}^{\D,\D} ([\D, 1\ot M_x])\|_{ \mcL_{d,\8}}  \\
&\leq C_d \| T_{\p_3}^{\D,\D} ([\D, 1\ot M_x])\|_{ \mcL_{d,\8}},
\end{split}\ee
where the constant $C_d>0$ does not depend on $x$. Since $\p_3(\lambda,\mu)=(1+\lambda^2)^{-1/4} (1+\mu^2)^{-1/4}$ is a product a function of $\lambda$ and a function of $\mu$, by \eqref{doi-def}, we have
$$T_{\p_3}^{\D,\D} ([\D, 1\ot M_x]) = (1+\D^2)^{-1/4} [\D, 1\ot M_x](1+\D^2)^{-1/4}. $$
Hence 
$$\|[g(\D), 1\ot M_x]\|_{\mcL_{d,\8}} \leq C_d \| (1+\D^2)^{-1/4} [\D, 1\ot M_x](1+\D^2)^{-1/4}\|_{\mcL_{d,\8}}.$$
Expanding out $\D$ and using the quasi-triangle inequality for $\mcL_{d,\8}$, we have
\be\begin{split}
&\| (1+\D^2)^{-1/4} [\D, 1\ot M_x](1+\D^2)^{-1/4}\|_{\mcL_{d,\8}}   \\
&\leq K_d \sum_{j=1}^d \| (1+\D^2)^{-1/4} [\g_j\ot D_j, 1\ot M_x](1+\D^2)^{-1/4}\|_{\mcL_{d,\8}},
\end{split}\ee
where $K_d>0$ depends only on $d$. But $[\g_j\ot D_j, 1\ot M_x] =-\ri  \g_j \ot M_{\pd_j x} $, thus we obtain
$$\| (1+\D^2)^{-1/4} [\g_j\ot D_j, 1\ot M_x](1+\D^2)^{-1/4}\|_{\mcL_{d,\8}} = \| (1-\Delta)^{-1/4} M_{\pd_j x}(1- \Delta)^{-1/4}\|_{\mcL_{d,\8}}.$$
Note that the first norm $\|\cdot\|_{\mcL_{d,\8}}$ is the norm of $\mcL_{d,\8}(\com^N\otimes L_2(\qt))$, and the second one is the norm of $\mcL_{d,\8}(L_2(\qt))$.

We are reduced to estimating the quantity $\| (1-\Delta)^{-1/4} M_{\pd_j x}(1- \Delta)^{-1/4}\|_{\mcL_{d,\8}}$. By polar decomposition, for every $j$, there is a partial isometry $U_j$ such that 
$$\pd_j x  =  U_j |\pd_j x | =  U_j |\pd_j x |^{\frac 1 2 } |\pd_j x |^{\frac 1 2 }.$$
Taking $\beta=\frac 1 2 $, and recalling that $x$ is such that $\|U_j |\pd_j x|^{\frac 1 2 } \|_{2d}\leq \|\, |\pd_j x|^{\frac 1 2 } \|_{2d} = \|\pd_j x \|_d^{\frac 1 2 }<\8$, we apply Corollary \ref{Cwikel-type-cor} to get
(for some constant $Q_d$)
$$\| M_{|\pd_j x|^{\frac 1 2}}(1- \Delta)^{-1/4}\|_{\mcL_{2d,\8}}= \| (1- \Delta)^{-1/4} M_{|\pd_j x|^{\frac 1 2}}\|_{\mcL_{2d,\8}}\leq Q_d \|\, |\pd_j x|^{\frac 1 2 } \|_{2d}$$
and 
$$\| (1- \Delta)^{-1/4} M_{U_j |\pd_j x|^{\frac 1 2}}\|_{\mcL_{2d,\8}}\leq Q_d \|U_j |\pd_j x|^{\frac 1 2 } \|_{2d}\leq  Q_d \|\, |\pd_j x|^{\frac 1 2 } \|_{2d}.$$
Thus, by the H\"{o}lder inequality \eqref{weak-type-holder}, 
$$\| (1-\Delta)^{-1/4} M_{\pd_j x}(1- \Delta)^{-1/4}\|_{\mcL_{d,\8}}\leq c_{\frac{d}{2},\frac{d}{2}}Q_d^2\|\, |\pd_j x|^{\frac 1 2 } \|_{2d}^2=c_{\frac{d}{2},\frac{d}{2}}Q_d^2\|\pd_j x \|_d.$$
Taking $B_d = c_{\frac{d}{2},\frac{d}{2}}d Q_d^2 C_d K_d$, we conclude that 
\begin{equation}\label{g(D)_sobolev_bound}
    \|[g(\D), 1\ot M_x]\|_{\mcL_{d,\8}} \leq B_d \sum_{j=1}^d\|\pd_j x \|_d\leq B_d \|x\|_{\dot H_d^1}.
\end{equation}

{\hl 
We now remove the initial assumption that $x \in C^\infty(\qt)$. Suppose that $x \in \dot H_d^1(\qt)$. As $C^\infty(\qt)$ is dense in $H_d^1(\qt)$ \cite[Proposition 2.7]{XXY2018}, we may select a sequence $\{x_n\}_{n=0}^\infty \subset C^\infty(\qt)$ 
such that $\lim_{n\to\infty} \|x_n-x\|_{H_d^1(\qt)} = 0$. From \eqref{g(D)_sobolev_bound}, we have that the sequence $\{[g(D),1\otimes M_{x_n}]\}_{n=0}^\infty$ is Cauchy in the $\mcL_{d,\infty}$ topology. Hence there is a limit $T \in \mcL_{d,\infty}$. 
On the other hand, if $\eta \in C^\infty(\qt)\otimes \com^N$ from the H\"older inequality we have:
\begin{equation*}
    \|(1\otimes M_{x_n})\eta-(1\otimes M_x)\eta\|_{L_2(\qt)\otimes \com^N} \leq \|x_n-x\|_{L_d(\qt)}\|\eta\|_{L_{2d/(d-2)}(\qt)\otimes \com^N}
\end{equation*}
and similarly,
\begin{equation*}
    \|(1\otimes M_{x_n})g(D)\eta-(1\otimes M_x)g(D)\eta\|_{L_2(\qt)\otimes \com^N} \leq \|x_n-x\|_{L_d(\qt)}\|g(D)\eta\|_{L_{2d/(d-2)}(\qt)\otimes \com^N}.
\end{equation*}
Therefore for each fixed $\eta \in C^{\infty}(\qt)\otimes \com^N$ we have:
\begin{equation*}
    [g(D),1\otimes M_{x_n}]\eta \rightarrow [g(D),1\otimes M_x]\eta
\end{equation*}
in the $L_2(\qt)\otimes \com^N$ sense. Thus,
\begin{equation*}
    [g(D),1\otimes M_{x}]\eta = T\eta
\end{equation*}
for all $\eta \in C^{\infty}(\qt)\otimes \com^N$. Therefore $T$ and $[g(D),1\otimes M_x]$ are equal, and we have:
\begin{equation*}
    [g(D),1\otimes M_{x_n}]\to [g(D),1\otimes M_x]
\end{equation*} 
in the $\mcL_{d,\infty}$ topology. Thus \eqref{g(D)_sobolev_bound} holds for all $x \in \dot H^1_d(\qt)$.
}
\end{proof}

Now we are able to complete the proof of Theorem \ref{sufficiency}.

\begin{proof}[Proof of Theorem \ref{sufficiency}]
Let $x\in \dot H_d^1(\qt)$. Combining Lemmas \ref{Cwikel-xp} and \ref{commutator-Sob}, we find that
$$ \big\|\big[\sgn(\D)   , 1\ot M_x  \big]\big\|_{\mcL_{d,\8}}  \leq C_{d,d}  \|x\|_d  + B_d \|x\|_{\dot H_d^1} .       $$
We can remove the dependence on $\|x\|_d$ on the right hand side by the aid of the Poincar\'e inequality \eqref{Poincare}. Since for constant operator $\widehat{x}(0)\in L_\8(\qt)$, it is obvious that $\big[\sgn(\D)   , 1\ot M_{\widehat{x}(0)}  \big]=0$, we have
\be\begin{split}
\big\|\big[\sgn(\D)   , 1\ot M_x  \big]\big\|_{\mcL_{d,\8}}  & = \big\|\big[\sgn(\D)   , 1\ot M_{x-\widehat{x}(0)}  \big]\big\|_{\mcL_{d,\8}}  \\
&\leq C_{d,d}  \|x-\widehat{x}(0) \|_d  + B_d \|x-\widehat{x}(0) \|_{\dot H_d^1}\\
&\leq C_d \|x\|_{ \dot H_d^1}.
\end{split}\ee
The theorem is therefore proved.
\end{proof}

\section{Pseudodifferential operators on quantum tori}\label{sec-pdo}

In this section we give an introduction to some recent developments in pseudodifferential operators on quantum tori. The most important result stated in this section for us is Theorem \ref{Connes-qt}, which 
is a form of Connes' trace formula obtained in \cite{MSZ2018}.

The theory of pseudodifferential operators goes back to Kohn-Nirenberg  \cite{KN1965} and H\"{o}rmander \cite{Hor1965}. It has been extended to the noncommutative setting, especially the quantum torus case, by many authors; see for instance \cite{MPR2005,LMR2010,LJP2016,GJP2017,
Tao2018,XX2018}. Our main references of this part are \cite{Connes1980,Baaj1988} and \cite{CT2011}, while the details can be found in \cite{HLP2018,HLP2018a}. In the following, let us collect some definitions and well known properties of symbol classes and pseudodifferential operators on quantum tori.

Denote by $\Jx$ the function $(1+|\xi|^2) ^{\frac 1 2 }$ on $\real^d$.
For every $m\in \real$, the class $S^m(\real^d;  \cS(\qt))$ consists of all maps $\rho \in C^{\8}(\real^d ; \cS(\qt))$ such that, for all multi-indices $\alpha,\beta \in \nat_0^d$, there exists $C_{\alpha, \beta} >0$ such that 
$$\|D^\alpha D_\xi^\beta  \rho(\xi) \| \leq C_{\al,\bt}  \Jx ^{m-|\bt|_1} ,\quad \forall \xi \in \real^d . $$ 
Endowed with the locally convex topology generated by the semi-norms
$$p_N^{(m)} (\rho)  : =   \sup_{|\al|_1 +|\bt|_1 \leq N   }  \sup _{\xi \in \real ^d}   \Jx ^{-m+|\bt|_1}  \|D^\alpha D_\xi^\beta  \rho(\xi) \| ,\quad N\in \nat _0, $$
$S^m(\real^d;  \cS(\qt))$ is then a Fr\'echet space.

Let $\rho \in S^m(\real^d;  \cS(\qt))$, $m\in \real$, and $\rho_j(\xi) \in S^{m-j}(\real^d;  \cS(\qt))$ for each $j\in \nat$. If for every $N\geq 1$, 
$$\rho(\xi)  - \sum_{j<N}\rho_j(\xi) \in S^{m-N}(\real^d;  \cS(\qt)),$$
we shall write $\rho(\xi) \sim \sum_{j\geq 0}  \rho_j(\xi)$. This is referred to as an asymptotic expansion of the symbol $\rho$.

The homogeneous class of symbols $\dot S^m(\real^d; \cS(\qt))$ consists of maps $\rho \in C^{\8}(\real^d \setminus \{0\} ; \cS(\qt))$ satisfying
$$\rho (\lambda \xi )    =  \lambda^m \rho(\xi),\quad \forall \xi \in \real^d \setminus \{0\},\;\forall \lambda >0.$$
In this case, $\rho$ on $\real^d \setminus \{0\}$ is determined by its restriction to $\mathbb S^{d-1}$, the $d$-dimensional unit sphere.
If {\hl a (not necessarily homogeneous) symbol} $\rho$ admits an asymptotic expansion $\rho \sim \sum_{j\geq 0}  \rho_{m-j}$ with $ \rho_{m-j}\in \dot S^{m-j}(\real^d; \cS(\qt))$ for each $j\geq 0$, then $\rho$ is called a {\it classical }symbol, and the leading term $\rho_m$ is called the principal symbol of $\rho$.

\medskip

Let us turn to the definition of pseudodifferential operators with the above symbols on quantum torus. Let $\al _s$ be a $d$-parameter group of automorphisms given by
 \beq\label{action-R-qt}
 \al_s (U^n)= e^{2\pi \ri s\cdot n}  U^n,\eeq
 which is a periodic version of the action in \eqref{action-T-qt} if we identify $[0,1]^d $ with $\T^d$ by the correspondence $(s_1,s_2,\cdots, s_d) \leftrightarrow (e^{2\pi \ri s_1}, e^{2\pi \ri s_2}, \cdots, e^{2\pi \ri s_d})$. For $\rho \in C^{\8}(\real^d ; \cS(\qt))$, let $P_\rho$ be the pseudodifferential operator sending arbitrary $a \in \cS(\qt)$ to 
\beq\label{pdo-def}
P_\rho(a) := \int_{\real^d}  \int _{\real^d}   e ^{-2\pi \ri s \cdot \xi } \rho (\xi) \al_s(a) ds \, d\xi .\eeq
Note that this integral does not converge absolutely; it is defined as an oscillatory integral. See \cite{CT2011,HLP2018,HLP2018a,Tao2018} for more information. By \cite[Proposition~5.9]{HLP2018}, if $a = \sum_{n\in \ent^d} a_n U^n\in \cS(\qt)$ and $\rho \in S^m(\real^d ; \cS(\qt))$ with {\hl $m\in \real$}, then 
\beq\label{pdo-def-ent}
P_\rho (a)  = \sum_{n\in \ent^d}   \rho(n) a_n U^n ,\eeq
where the sum converges in the operator norm to an element in $\cS(\qt)$.
In other words, the pseudodifferential operator on $\qt$ with symbol $\rho$ is determined by the value of $\rho$ on $\ent^d$, which coincides with the definition given in \cite{LJP2016}.

If $\rho \in S^m(\real^d;\cS(\qt))$ with {\hl $m \in \real$}, then $P_\rho$ is said to be a pseudodifferential operator
of order $m$.

Also note that, by the noncommutativity, if we change the order of $\rho(\xi)$ and $\al_s(a)$ in \eqref{pdo-def} (or $\rho(n) $ and $U^n$ in \eqref{pdo-def-ent}), we get another pseudodifferential operator with the same symbol. In \cite{XX2018}, these two operators are distinguished as column and row operators. But in this paper, we will not need to consider both kinds of operators, and so we focus only on those with the form  \eqref{pdo-def} or \eqref{pdo-def-ent}.

\begin{ex}\label{ex-symbol}
Let us formulate some first examples of symbols defined above.
\begin{enumerate}[{\rm i)}]
\item Let $x\in \cS( \T_\theta^d)$ and consider the constant function $\psi (\xi) \equiv x$, $\xi \in \real^d$. Obviously, $\psi \in S^0(\real^d; \cS(\qt))$. So by the above definition, the multiplier $M_x (y) =xy$ on $ \T_\theta^d $ is an order $0$ pseudodifferential operator. The principal symbol of this operator is $x$ itself.
\item Let $k\in \nat_0^d$. The symbol of the $|k|_1$-order differential operator $D^k= D_1^{k_1}\cdots D_d^{k_d}$ is $\psi (\xi) =  (2\pi \xi_1)^{k_1}(2\pi \xi_2)^{k_2}\cdots (2\pi \xi_d)^{k_d}$. It is easily checked that $\psi \in S^{|k|_1}(\real^d; \cS(\qt))$. Thus $D^k$ is a pseudodifferential operator of order $|k|_1$, and its principal symbol is $ (2\pi \xi)^k$.
\item Let $\alpha \in \real$, and consider the $\alpha$-order Bessel potential $J^\alpha= (1-  \Delta )^{\frac{\al}{2}}$ on the quantum torus, which is a Fourier multiplier with symbol $\psi (\xi) = \lan 2\pi \xi \ran ^\alpha = (1+|2\pi \xi|^2) ^{\frac \al 2 } \in S^\alpha(\real^d; \cS(\qt))$. Thus, $J^\alpha$ is an $\alpha$-order pseudodifferential operator. Moreover, as a scaler-valued function, $\psi (\xi) $ has the asymptotic expansion 
$$\lan 2\pi \xi \ran ^\alpha \sim \sum_{j=0}^{\8}  \begin{pmatrix}
j \\ 
\frac{\al }{2}
\end{pmatrix}  |2\pi \xi|^{\alpha-2j} .$$
Hence, $J^\alpha$ is classical with principal symbol $|2\pi \xi |^\alpha$. See \cite[Proposition~5.14]{HLP2018}.
\end{enumerate}
\end{ex}

The above examples illustrate that both pointwise multipliers $M_x$ and Fourier multipliers $T_g$ from \eqref{def-Fourier-Z} are considered as the special cases of pseudodifferential operators. For general symbol $\rho$, $P_\rho$ may be thought as a limit of linear combinations of operators composed by pointwise multipliers and Fourier multipliers.

\medskip

The composition of two pseudodifferential operators is again a pseudodifferential operator, and there is a method for computing an asymptotic expansion of its symbol.
The following proposition, which is the quantum analogue of the classical result in \cite[p. 237]{Stein1993}, first appears in \cite{CT2011}; a complete proof is given in \cite[Proposition~7.5]{HLP2018a}.

\begin{prop}\label{symbol-composition}
Let $\rho_1$, $\rho_2$ be two symbols in $S ^{n_1}(\real^d; \cS(\qt))$ and $S ^{n_2}(\real^d; \cS(\qt))$ respectively. Then there exists a symbol $\rho_3$ in $S ^{n_1+n_2}(\real^d; \cS(\qt))$ such that 
$$
P_{\rho_3} =P_{\rho_1}   P_{\rho_2} .
$$
Moreover, 
\begin{equation}\label{composition-asym}
\rho_3-\sum_{|\al |_1<N_0} \frac{(2\pi {\rm i})^{-|\al |_1}}{\al  !}D_\xi^\al  \rho_1 D^\al  \rho_2 \in S^{n_1+n_2- N_0}(\real^d; \cS(\qt)), \quad \forall\, N_0\geq 0,
\end{equation}
where the first derivative $D_\xi^\al $ is the derivative of $\rho_1$ with respect to the variable $\xi \in \real^d$, and the second derivative $D^\al $ is the derivation on $\cS(\qt)$ described in Section \ref{calculus definition subsubsection}.
\end{prop}

 {\hl   Many authors have considered the question of the mapping properties of pseudodifferential operators on functions spaces on quantum tori \cite{XX2018,GJP2017,HLP2018a}. In this paper we are concerned solely with the boundedness
    of a pseudodifferential operator on $L_2(\qt)$.
    
    The following proposition can be found in \cite[Proposition 10.1]{HLP2018a}, \cite[Corollary 6.6]{Tao2018}.    
    \begin{prop}\label{bdd-2-Sobolev}
        Let $\rho\in S^0(\real^d; \cS(\qt))$. Then the pseudodifferential operator $P_\rho$ extends to a bounded operator from $L_2(\T_\theta^d)$ to $L_2(\T_\theta^d)$.
    \end{prop}
    Proposition \ref{bdd-2-Sobolev} is simply a special case of the general Sobolev space mapping property of pseudodifferential operators \cite[Proposition 6.6]{HLP2018a}.
    Even greater generalisations to mapping properties of pseudodifferential operators on Sobolev spaces and Besov and Triebel-Lizorkiin spaces \cite[Section 6.2]{XX2018} are also known.

    Symbols of negative order are in particular of order zero, and thus if $m > 0$ and $\rho \in S^{-m}(\real^d,\cS(\qt))$ then $P_{\rho}$ has bounded extension on $L_2(\qt)$. However in the case of strictly negative order we can provide
    more detailed information on $P_{\rho}$. The following is proved in \cite[Lemma~13.6]{HLP2018a}:
    }
    \begin{prop}\label{pdo-weak-L}
        If $\rho\in S^{-m}(\real^d; \cS(\qt))$ with $m>0$, then $P_\rho$ is a compact operator on $L_2(\T_\theta^d)$. Furthermore, $P_\rho \in \mathcal{L}_{\frac d m,\8}$.
    \end{prop}
    {\hl 
        The proof of Proposition \ref{pdo-weak-L} is a simple combination of the fact that since $P_{\rho}$ has order $-m$, and $J^m$ has order $m$, the product formula in Proposition \ref{symbol-composition} implies
        that the composition $P_{\rho}J^m$ is of order zero. Hence by Proposition \ref{bdd-2-Sobolev}, $P_{\rho}J^{m}$ has bounded extension, and since $J^{-m} \in \mathcal{L}_{\frac{d}{m},\infty}$ \eqref{bessel potential ideal}, it follows immediately that $P_{\rho} \in \mathcal{L}_{\frac{d}{m},\infty}$.

%
%

    Thanks to Proposition \ref{pdo-weak-L}, we can easily obtain from the symbol calculus the following:
    \begin{cor}\label{commutator-MxJ}
        Let $x\in \cS(\qt)$, and $\al>0$. Then 
        $$[M_x , (1-\Delta )^{-\frac{\al}{2}}] \in \mathcal{L}_{\frac{d}{\al+1} ,\8}.$$
    \end{cor}
    Indeed, $[M_x , (1-\Delta )^{-\frac{\al}{2}}] $ is a pseudodifferential operator of order at most $-\al-1$, as can be seen by a short computation using Proposition \ref{symbol-composition}.
    }
%
%

\medskip

Next, we are going to present Connes' trace formula on quantum torus in the specific form obtained in \cite[Theorem 6.5]{MSZ2018}. This trace formula will play a crucial role in the proof of the trace formula for a quantised differential. Recall that if $\rho$ is a homogeneous symbol of order $0$, then $\rho(\xi) =\rho(\frac{\xi}{|\xi|}) $ for every $\xi \neq 0$. So this $\rho$ could be viewed as a function on the $(d-1)$-dimensional sphere $\mathbb{S}^{d-1}$.
\begin{thm}\label{Connes-qt}
Let $A$ be a classical pseudodifferential operator on $\qt$ of order $0$ with self-adjoint extension, and denote by $\rho_A$ its principal symbol. Then for any normalised trace $\vf$ on $\mathcal{L}_{1,\8}$, we have
$$\vf \big(|A|^d(1-\Delta )^{-\frac d 2}\big) =\frac 1 d   \int_{\mathbb{S}^{d-1}}          \tau(|\rho_A(s)|^d) ds   .$$
\end{thm}
The reason to refer specifically to \cite{MSZ2018} is that if $d$ is odd then $|A|^d$ is not a pseudodifferential operator in the usual sense, and so it needs to be understood as an element
of the $C^*$-closure of the algebra of order $0$ pseudodifferential operators on $L_2(\qt)$. It is proved in \cite{MSZ2018} that on the $C^*$-closure the principal symbol mapping extends
to a $C^*$-algebra homomorphism, and hence $\rho_{|A|^d} = |\rho_A|^d$.

\section{The trace formula}\label{sec-Ctf}
    This section is devoted to the proofs of Theorem \ref{trace formula} and Corollary \ref{trace formula-bound}. That is, we show that for all $x \in \dot{H}^1_d(\qt)$ and all
    continuous normalised traces $\varphi$ on the ideal $\mcL_{1,\8}$ that:
    \begin{equation}\label{trace formula in section}
            \varphi(|\qd x|^d) = c_d\,\int_{\mathbb{S}^{d-1}} \tau\Bigg(\Big(\sum_{j=1}^d |\pd_j x-s_j\sum_{k=1}^d s_k\pd_k x|^2\Big)^{\frac{d}{2}}\Bigg)\,ds
    \end{equation}
    for a positive constant $c_d$. Moreover, there are positive constants $0 < c_d < C_d < \infty$ such that:
    \begin{equation}\label{trace formula-bound in section}
        c_d \| x\|_{\dot{H}_d^1}^d \leq  \varphi(|\qd x|^d) \leq C_d \| x\|_{\dot{H}_d^1}^d .
    \end{equation}
    Our strategy of proof is as follows: first, \eqref{trace formula in section} is proved for $x \in \cS(\qt)$ by aid of the theory of pseudodifferential operators developed in the preceding section. 
    Then by an approximation argument based on the density of $\cS(\qt)$ in $\dot{H}^{1}_d(\qt)$, we complete the proof of \eqref{trace formula in section} in full generality.
        Finally \eqref{trace formula-bound in section} is achieved by bounding the right hand side of \eqref{trace formula in section} from above and below by a constant multiple of $\|x\|_{\dot{H}^1_d}^d$. 
    

    To begin with, we explain how the operator $|\qd x|^d$ can, up to trace class perturbations, be written in the form $|A|^d(1+\D^2)^{-\frac{d}{2}}$ for a certain order zero pseudodifferential operator $A$.
    Let $x\in \cS(\qt)$. For $j=1,\cdots, d$, we define the operators $\{A_j\}_{j=1}^d$ on $L_2(\qt)$ by 
    $$A_j\eta := \left(M_{\pd_j x } -\frac 1 2  \sum_{k=1}^d  \Big( \frac{D_j D_k}{1-\Delta}  M_{\pd_k x }   + M_{\pd_k x }  \frac{D_j D_k}{1-\Delta} \Big)\right)\eta,\quad \eta \in L_2(\qt).$$
    For each $j$, $A_j$ is defined initially on $\cS(\qt)$, but by functional calculus $A_j$ extends uniquely to a bounded operator on $L_2(\qt)$ which we denote with the same symbol. We then define the operator $A$ on $\com^N  \otimes L_2(\qt)$ as 
    \beq\label{def-A}A:= \sum_{j=1}^d    \gamma_j \otimes A_j . \eeq

    If $x= x ^*$, since $\pd_j$ commutes with the adjoint operation $*$, we have for every $y_1, y_2 \in L_2(\qt)$,
    \be
    \lan (\pd_j x) y_1 , y_2 \ran = \tau \big((\pd_j x ) y_1 y_2^*\big) = \tau \big(y_1(\pd_j x^* y_2)  ^*\big) = \lan  y_1 , (\pd_j x^*)y_2 \ran,
    \ee
    which yields $(M_{\pd_j x })^* = M_{\pd_j x^*} = M_{\pd_j x}$. Furthermore, since each $D_j$ is a self-adjoint operator on $ L_2(\qt)$, we know that 
    $$\Big(\frac{D_j D_k}{1-\Delta}  M_{\pd_k x } \Big)^*  = M_{\pd_k x }  \frac{D_j D_k}{1-\Delta} .$$
    Therefore, each $A_j$ is a self-adjoint operator, and so is $A$.

    We will now show that $|\qd x|^d-|A|^d(1+\D^2)^{-\frac{d}{2}} \in \mcL_1$. The following lemma is an important first step:
    \begin{lem}\label{commutator-AJ}
    Let $\bt \geq 0$, and $\al >0$ be such that $\al+1 <d$. Then for $A$ defined in \eqref{def-A}, we have
    $$[|A|,  (1+\D^2 ) ^{-\frac \al 2} ] \,  ( 1+\D^2)^{-\frac{\bt}{2}}\in \mathcal{L}_{\frac{d}{\al+\bt +1} , \8 }. $$
    \end{lem}
    \begin{proof}
        As {\hl $(1+\D^2)^{-\frac{1}{2}} \in \mcL_{d,\infty}$}, the operator $ ( 1+\D^2)^{-\frac{\bt}{2}}$ is in $\mathcal{L}_{\frac{d}{\bt} , \8 }$. Thus by H\"older's inequality, it suffices to consider $\bt =0$. First, we shall prove that $[A_j ,  (1-\Delta)^{-\frac{\al}{2} }   ]       \in \mathcal{L}_{\frac{d}{\al  +1} , \8 }$. From Corollary \ref{commutator-MxJ}, we have that 
    $$[M_{\pd_j x } ,  (1-\Delta)^{-\frac{\al}{2} }   ]       \in \mathcal{L}_{\frac{d}{\al  +1} , \8 } .$$
    Hence, by linearity, it suffices to prove that 
    \beq\label{commutator-DMJ}
    \Big[  \frac{D_j D_k}{1-\Delta}M_{\pd_k x } ,  (1-\Delta)^{-\frac{\al}{2} }   \Big]       \in \mathcal{L}_{\frac{d}{\al  +1} , \8 } \eeq
    and 
    \beq\label{commutator-MDJ}
    \Big[  M_{\pd_k x }  \frac{D_j D_k}{1-\Delta} ,  (1-\Delta)^{-\frac{\al}{2} }   \Big]       \in \mathcal{L}_{\frac{d}{\al  +1} , \8 } .\eeq
    {\hl Note that \eqref{commutator-MDJ} follows from \eqref{commutator-DMJ} by taking the adjoint. So we prove only \eqref{commutator-DMJ}.}
    However, since $\frac{D_j D_k}{1-\Delta}$ commutes with $ (1-\Delta)^{-\frac{\al}{2} }  $, we have
    $$\Big[  \frac{D_j D_k}{1-\Delta}M_{\pd_k x } ,  (1-\Delta)^{-\frac{\al}{2} }   \Big]    = \frac{D_j D_k}{1-\Delta}\Big[  M_{\pd_k x } ,  (1-\Delta)^{-\frac{\al}{2} }   \Big]. $$
    By functional calculus, $\frac{D_j D_k}{1-\Delta}$ is bounded on $L_2(\qt)$. Then \eqref{commutator-DMJ} and \eqref{commutator-MDJ} follow from Corollary \ref{commutator-MxJ} and the boundedness of $\frac{D_j D_k}{1-\Delta}$ on $L_2(\qt)$.

    Thus, we have proved that 
    $$[A,  (1+\D^2)^{-\frac{\al}{2} }   ]   =\sum_{j=1}^d   \gamma_j \otimes [A_j ,  (1-\Delta)^{-\frac{\al}{2} }   ]       \in \mathcal{L}_{\frac{d}{\al  +1} , \8 }.$$
    To complete the proof, we need to replace $A$ with $|A|$. 
    To this end, we use the result of \cite{PS2011}, which implies that if $1<p<\8$ and $A$ and $B$ are self-adjoint operators such that $[A, B] \in \mathcal{L}_p$, then $[|A|, B] \in\mathcal{L}_p$; see also \cite[Corollary~3.5]{DDPS1997} for more general results. Since $\frac{d}{\al +1} >1$, the result follows from interpolation.
    \end{proof}

\begin{lem}\label{T-AJ}
Let $T$ be a bounded operator on $\com^N \otimes L_2(\qt)$, and suppose that 
$$T\in A(1+\D^2)^{-\frac 1 2 }    +\mathcal{L}_{\frac{2d}{3}, \8},  $$
where $A$ is given in \eqref{def-A}.
Then $|T|^d   \in \mathcal{L}_{1,\8}$ and for any continuous normalised trace $\vf$ on $\mathcal{L}_{1,\8}$, we have 
$$\vf (|T|^d) =\vf (|A|^d     (1+\D^2)^{-\frac d 2}  ).$$
\end{lem}
\begin{proof}
By the aid of Lemma \ref{commutator-AJ}, the proof proceeds as in \cite [Lemma~14]{LMSZ2017}, and is therefore omitted.
\end{proof}

\begin{lem}\label{commutator-Dx}
For $x\in \cS(\qt)$ and $A$ defined in \eqref{def-A}, we have
    $$\qd x - A(1+\D^2)^{-\frac 1 2} = [\sgn(\D), 1\otimes M_x]  - A(1+\D^2)^{-\frac 1 2}   \in  \mathcal{L}_{\frac{d}{2},\8}.$$
\end{lem}
\begin{proof}
Let $g(\D) = \D (1+\D^2) ^{-\frac 1 2 }$. Then 
\be\begin{split}
\sgn(\D) - g(\D) &= \sgn (\D) \Big(1 -\frac{|\D|}{(1+\D^2)^{\frac 1 2 }}\Big)\\
&= \sgn (\D) \Big(\frac{1}{(1+\D^2)^{\frac 1 2 } \big(   (1+\D^2)^{\frac 1 2 }  +|\D|  \big)     }\Big).
\end{split}\ee
Since $(1+\D^2)^{-\frac 1  2 }  \in  \mathcal{L}_{d,\8}$, it follows that $\sgn(\D) - g(\D) \in \mathcal{L}_{\frac d  2  , \8}$. Therefore,
$$[\sgn(\D), 1\otimes M_x] - [g(\D), 1\otimes M_x] \in \mathcal{L}_{\frac d 2 ,\8} .$$
Thus, it suffices to prove 
\beq\label{commutator-gDx}
[g(\D), 1\otimes M_x]  - A(1+\D^2)^{-\frac 1 2}   \in  \mathcal{L}_{\frac{d}{2},\8}.\eeq

Now let us prove \eqref{commutator-gDx}. {\hl By a short computation using Proposition \ref{symbol-composition}}, we see that the principal symbol of $[\frac{D_j}{(1-\Delta)^{\frac{1}{2}}},   M_x] $ is 
\beq\label{ps-DM}
\frac{1}{|2\pi \xi| } \pd_j x  - \sum_{k=1}^d  \frac{2\pi \xi_k  2\pi \xi_j}{|2\pi\xi|^3}  \pd_k x\,.\eeq
We also need to determine the principal symbol of $A_j  (1-\Delta)^{-\frac 1 2 } $, and to this end we compute the principal symbol of $A_j$.
Recall that 
$$A_j = M_{\pd_j x } -\frac 1 2  \sum_{k=1}^d  \Big( \frac{D_j D_k}{1-\Delta}  M_{\pd_k x }   + M_{\pd_k x }  \frac{D_j D_k}{1-\Delta} \Big).$$
It is evident that the symbol of $M_{\pd_k x }  \frac{D_j D_k}{1-\Delta} $ is $\frac{2\pi\xi_j 2\pi \xi_k }{1+|2\pi\xi|^2 } \pd_k x$, so the principal symbol is $\frac{\xi_j \xi_k }{|\xi| ^2 } \pd_k x$.
By Proposition \ref{symbol-composition}, we know that 
the symbol of $\frac{D_j D_k}{1-\Delta}  M_{\pd_k x } $ has the asymptotic expansion 
$$\sum _{\alpha \in \mathbb{N}_0^d}  \frac{(2\pi {\rm i})^{-|\al |_1}}{\al  !}D_\xi^\al \Big(\frac{2\pi\xi_j 2\pi \xi_k   }{ (1+|2\pi\xi|^2)^{\frac 1 2 } }\Big) D^\al  (\pd_k x) . $$
Thus, the principal symbol of $\frac 1 2  \sum_{k=1}^d  \Big( \frac{D_j D_k}{1-\Delta}  M_{\pd_k x }   + M_{\pd_k x }  \frac{D_j D_k}{1-\Delta} \Big)$ is $  \sum_{k=1}^d  \frac{\xi_k \xi_j}{|\xi|^2}  \pd_k x$, which ensures that the principal symbol of $A_j(1-\Delta) ^{-\frac 1 2 }$ is of order $-1$, given by
$$\frac{1}{2\pi |\xi| } \pd_j x  - \sum_{k=1}^d  \frac{\xi_k \xi_j}{2\pi |\xi|^3}  \pd_k x\,,$$
 the same as that of $[\frac{D_j}{(1-\Delta)^{\frac{1}{2}}},   M_x] $ given in \eqref{ps-DM}. Hence, the order of $[\frac{D_j}{(1-\Delta)^{\frac{1}{2}}},   M_x]  - A_j (1-\Delta)^{-\frac 1 2}  $ is $-2$. By Theorem \ref{pdo-weak-L}, we have 
\be
[\frac{D_j}{(1-\Delta)^{\frac{1}{2}}},   M_x]  - A_j (1-\Delta)^{-\frac 1 2}   \in  \mathcal{L}_{\frac{d}{2},\8}.\ee
Since $[g(\D), 1\otimes M_x]  = \sum_j \gamma_j \otimes [\frac{D_j}{(1-\Delta)^{\frac{1}{2}}},   M_x]   $ and $A(1+\D^2)^{-\frac 1 2} =  \sum_j \gamma_j \otimes  A_j (1-\Delta)^{-\frac 1 2}   $, we obtain \eqref{commutator-gDx}. The lemma is thus proved.
\end{proof}

Based on the above lemmas, we are able to complete the proof of Theorem \ref{trace formula}.

\begin{proof}[Proof of Theorem \ref{trace formula}]
    Assume initially that $x\in \cS(\qt)$. By Lemma \ref{commutator-Dx}, we have that:
    \begin{equation*}
        \qd x\in  A(1+\D^2)^{-\frac{1}{2}} + \mcL_{\frac{d}{2},\8}.
    \end{equation*}
For any continuous normalised trace $\vf$ on $\mcL_{1,\8}$, we invoke Lemma \ref{T-AJ} to obtain that:\
    \begin{equation*}
        \varphi(|\qd x|^d) = \varphi(|A|^d(1+\D^2)^{-\frac{d}{2}}).
    \end{equation*}
 In the proof of Lemma \ref{commutator-Dx}, we have that the principal symbol of $A_j$ is $  \pd_j x  - \sum_{k=1}^d  \frac{\xi_k \xi_j}{|\xi|^2}  \pd_k x$, which restricted to the unit sphere $\mathbb{S}^{d-1}$ is $  \pd_j x  - \sum_{k=1}^d   \xi_k \xi_j   \pd_k x$. Now we appeal to Theorem \ref{Connes-qt} to conclude
        \begin{equation*}
            \varphi(|\qd x|^d) = c_d\int_{\mathbb{S}^{d-1}} \tau\Bigg(\Big(\sum_{j=1}^d |\pd_j x-s_j\sum_{k=1}^d s_k\pd_k x|^2\Big)^{\frac{d}{2}}\Bigg)\,ds.
        \end{equation*}

    However the appeal to Theorem \ref{Connes-qt} relies on the assumption that $x \in \cS(\qt)$, so we remove this assumption by an approximation argument. Indeed, let $x\in \dot{H}_d^1 (\qt) $. By Theorem \ref{sufficiency}, we have $\qd x  \in \mcL_{d,\8}$. By the density of $\cS(\qt) $ in $\dot H_d^1 (\qt)$ (see \cite[Proposition~2.7]{XXY2018}), we can choose a sequence $\{x_n\}_{n=1}^\8 \subset \cS(\qt)$ such that $x_n \ra x$ in $\dot{H}_d^1(\qt)$. We shall show that $\varphi(|\qd x_n|^d) \ra \varphi(|\qd x|^d) $ and
    \beq\label{convergence-Tx}
        \int_{\mathbb{S}^{d-1}} \tau\Bigg(\Big(\sum_{j=1}^d |\pd_j x_n-s_j\sum_{k=1}^d s_k\pd_k x_n|^2\Big)^{\frac d 2}\Bigg)\,ds \ra \int_{\mathbb{S}^{d-1}} \tau\Bigg(\Big(\sum_{j=1}^d |\pd_j x-s_j\sum_{k=1}^d s_k\pd_k x|^2\Big)^{\frac d 2}\Bigg)\,ds.
    \eeq
    {\hl 
        Note that we have a bound:
        \begin{equation*}
             \int_{\mathbb{S}^{d-1}} \tau\Bigg(\Big(\sum_{j=1}^d |\pd_j x_n-s_j\sum_{k=1}^d s_k\pd_k x_n|^2\Big)^{\frac d 2}\Bigg)\,ds \leq C_d\|x_n\|_{\dot{H}^1_d}
        \end{equation*}
        for a certain constant $C_d$,
        and hence \eqref{convergence-Tx} is immediate. On the other hand, using Theorem \ref{sufficiency}, we have:
        \begin{equation*}
            \|\qd x-\qd x_n\|_{\mcL_{d,\8}} \leq C_d\|x-x_n\|_{\dot{H}^{1}_d(\qt)} \ra 0.
        \end{equation*}
        By a verbatim repetition of the argument in the proof of \cite[Theorem 17]{LMSZ2017}, we get 
        $$\||\qd x|^d-|\qd x_n|^d\|_{\mcL_{1,\8}} \ra 0.$$
        Since the trace $\varphi$ is assumed to be continuous in the $\mcL_{1,\infty}$ quasi-norm, it follows that $\varphi(|\qd x_n|^d)\ra \varphi(|\qd x|^d)$.
    }
\end{proof}

    We are now concerned with relating the right hand side of the trace formula in Theorem \ref{trace formula} with the $\dot H^1_d$-norm of $x$.

    \begin{proof}[Proof of Corollary \ref{trace formula-bound}]
   We prove the upper bound first. {\hl Denote
    \begin{equation*}
        T(x) := \Big(\sum_{j=1}^d |\partial_j x-s_j\sum_{k=1}^d s_k\partial_kx|^2\Big)^{\frac{1}{2}},\quad s \in \Sp^{d-1}.
    \end{equation*}
    Then 
    \begin{align*}
           | T(x)|^2 &= \sum_{j=1}^d \big|\partial_j x-s_j\sum_{k=1}^d s_k\partial_k x \big|^2\\
                    &= \sum_{j=1}^d\Big( |\partial_j x|^2 -\sum_{k=1}^d (s_j\partial_j x^* \cdot s_k\partial_k x+s_k\partial_k x^*\cdot s_j\partial_j x)+s_j^2|\sum_{k=1}^d s_k\partial_k x|^2\Big)\\
                    &= \sum_{j=1}^d |\partial_j x|^2 +|\sum_{k=1}^d s_k\partial_kx|^2-\sum_{j,k=1}^d \big(s_j\partial_j x^*\cdot s_k\partial_k x+s_k\partial_k x^*\cdot s_j\partial_j x\big).
    \end{align*}
    However, observing that
    \begin{equation*}
        |\sum_{j=1}^d s_j\partial_j x|^2 = \sum_{j,k=1}^d s_j\partial_j x^*\cdot s_k\partial_k x\,,
    \end{equation*}
    we get
    \beq\label{Tx-expand}
        |T(x)|^2 = \sum_{j=1}^d |\partial_j x|^2-|\sum_{j=1}^d s_j\partial_j x|^2.
    \eeq}
    We then have have easily:
        \begin{equation*}
            |T(x)|^2 \leq \sum_{j=1}^d |\partial_j x|^2.
        \end{equation*} 
        Therefore,
        \begin{equation*}
           \| |T(x)|^2 \|_{\frac d 2} \leq \|\sum_{j=1}^d |\partial_j x|^2 \|_{\frac d 2}.
        \end{equation*}
Hence, by \eqref{Sob-equi-norm}, for every $s\in \mathbb{S}^{d-1}$, we have
        \begin{equation*}
         \tau\big(   |T(x)|^d \big) =  \| |T(x)|^2 \|_{\frac d 2}^{\frac d 2} \leq \|\sum_{j=1}^d |\partial_j x|^2 \|_{\frac d 2} ^{\frac d 2 }\leq C_d\|x\|_{\dot{H}^1_d}^d .
        \end{equation*}
        Thus the upper bound is proved.
        
        Now we prove the lower bound. Since $|T(x)|^2 = \sum_{j=1}^d |\partial_j x-s_j\sum_{k=1}^d s_k\partial_kx|^2$, for each $j$ we have
        \begin{equation*}
            |\partial_j x-s_j\sum_{k=1}^d s_k\partial_k x|^2 \leq |T(x)|^2,
        \end{equation*}
        and therefore,
        \begin{equation}\label{Xj-leq-T}
            \|\partial_j x-s_j\sum_{k=1}^d s_k\partial_k x\|_d \leq \|T(x)\|_d.
        \end{equation}
        For brevity, define
        \begin{equation*}
            X_j = \|\partial_jx-s_j\sum_{k=1}^d s_k\partial_k x\|_d.
        \end{equation*}
Then \eqref{Xj-leq-T} implies that
        \begin{equation}\label{Xj-leq-T}
            \big(\sum_{j=1}^d X_j\big)^d \leq d^d\|T(x)\|_d^d.
        \end{equation}
        
        By the triangle inequality,
        \begin{align*}
            X_j &= \|(1-s_j^2)\partial_j x-\sum_{k\neq j} s_js_k\partial_kx\|_d\\
                &\geq (1-s_j^2)\|\partial_j x\|_d-\sum_{k\neq j} |s_js_k|\|\partial_kx\|_d.
        \end{align*}
        and therefore,
        \begin{equation*}
            \sum_{j=1}^d X_j \geq \sum_{j=1}^d\Big( (1-s_j^2)-\sum_{k\neq j} |s_js_k|\Big)\|\partial_j x\|_d.
        \end{equation*}
        Now, select $1\leq l\leq d$ such that $\|\partial_l x\|_{d}$ is the minimum of $\{\|\partial_1 x\|_d,\|\partial_2 x\|_d,\ldots,\|\partial_d x\|_d\}$. Denote by $e_l$ the $l$-th canonical basic vector of $\real^d$,
        and assume that $s \in B(e_l,\varepsilon)\cap \Sp^{d-1}$.
        We have:
        \begin{equation*}
           \big|(1-s_l^2)-\sum_{k\neq l}|s_ls_k|\,\big| \leq \max\big(1-s_l^2,\sum_{k\neq l}|s_ks_l |\big)
        \end{equation*} 
        Hence,
        $(1-s_l) \leq |s-e_l  | \leq \varepsilon$, so $(1-s_l^2) = (1-s_l)(1+s_l) \leq 2\varepsilon$, and by the Cauchy-Schwarz inequality
        \begin{align*}
            \sum_{k\neq l} |s_k s_l| &\leq (\sum_{k\neq l} |s_k|^2)^{\frac 1 2 }d^{\frac 1 2 }|s_l|\\
                                    &\leq |s-e_l |  d^{\frac 1 2 }  |s_l| \\
                                    &\leq \sqrt{d}\varepsilon.
        \end{align*}
        So,
        \begin{equation}\label{epsilon upper bound}
            \big|(1-s_l^2)-\sum_{k\neq l}|s_ls_k|\,\big| \leq \max\{2,\sqrt{d}\}\varepsilon.
        \end{equation}
        On the other hand, if $j\neq l$, then $|s_j| \leq \varepsilon$ and so:
        \begin{align}
            (1-s_j^2)-\sum_{k\neq j} |s_ks_j| &= 1-|s_j|\sum_{k=1}^d |s_k|\nonumber\\
                                              &\geq 1-\sqrt{d}\varepsilon  \label{epsilon lower bound}.
        \end{align}
        If we select $\varepsilon$ sufficiently small, we have $1-\sqrt{d}\varepsilon\geq 3\max\{2,\sqrt{d}\}\varepsilon$.
        Then combining \eqref{epsilon upper bound} and \eqref{epsilon lower bound}, we have that for all $j\neq l$:
        \begin{equation*}
            3\big|(1-s_l^2)-\sum_{k\neq l}|s_ls_k|\,\big| \leq (1-s_j^2)-\sum_{k\neq j} |s_ks_j|\,,
        \end{equation*}
        and thus,
        \begin{equation*}
           \big|(1-s_l^2)-\sum_{k\neq l} |s_ls_k|\,\big|   \, \|\partial_l x\|_{d}\leq \frac{1}{3}\big((1-s_j^2)-\sum_{k\neq j}|s_ks_j|\big)\|\partial_j x\|_d.
        \end{equation*}
Therefore, using the numerical inequality that if $|z| \leq \frac{1}{3}|w|$ then $|z-w| \geq \frac{2}{3}|w|$, we have
        \begin{align*}
           & \big((1-s_j^2)-\sum_{k\neq j}|s_ks_j|\big)\|\partial_j x\|_{d}+\big((1-s_l^2)-\sum_{k\neq l} |s_ls_k|\big)\|\partial_l x\|_{d}\\ &\geq \frac{2}{3} \big((1-s_j^2)-\sum_{k\neq j}|s_ks_j|\big)\|\partial_j x\|_{d}\\
  &\geq \frac{1}{3}(1-\sqrt{d}\varepsilon)(\|\partial_j x\|_{d}+\|\partial_l x\|_{d}).
        \end{align*}
Consequently, for $s \in B(e_l,\varepsilon)\cap \Sp^{d-1}$, we have,
        \begin{align*}
            \sum_{j=1}^d X_j \geq (1-\sqrt{d}\varepsilon)\sum_{j=1}^d \|\partial_j x\|_d.
        \end{align*}
Now,
        \begin{align*}
            \int_{\Sp^{d-1}} \Big(\sum_{j=1}^d X_j\Big)^d\,ds &\geq \int_{B(e_l,\varepsilon)\cap \Sp^{d-1}} \Big(\sum_{j=1}^d X_j\Big)^d\,ds\\
                                                                 &\geq c_{d,\varepsilon}\|x\|_{\dot{H}^1_d}.
        \end{align*}
By virtue of \eqref{Xj-leq-T}, the desired conclusion is proved.
    \end{proof}

\section{Proof of Theorem \ref{necessity}}\label{sec-nece}

In this section, we are going to give the proof of Theorem \ref{necessity}. We require a lemma on the quantised derivative of $x$ acting by a Fourier multiplier. 

    Recall that for a function $\psi \in L_1(\T^d)$, the convolution with $x \in {\hl L_2(\qt)}$ is defined as:
    \begin{equation*}
        \psi \ast x = \int_{\T^d} \psi(w)\al_{w^{-1}}(x)\,dm(w).
    \end{equation*}


\begin{lem}\label{commutator-convo}
    Let $\psi \in L_1(\T^d)$. If $x\in {\hl L_2(\qt)}$ is such that $\qd x$ {\hl extends to a bounded operator in $\mcL_{d,\8}$, then $\qd (\psi \ast x)$ also extends to a bounded operator in $\mcL_{d,\8}$ and we have:}
    $$\|\qd(\psi\ast x)\|_{\mcL_{d,\8}} \leq {\hl C_d}\|\qd x \|_{\mcL_{d,\8}} \|\psi\|_1 $$
    {\hl for a certain constant $C_d$.}
\end{lem}
\begin{proof}
Let $\al$ be the $d$-parameter group of automorphisms given in \eqref{action-T-qt}, i.e. if $u \in \T^d$ then $\al _u( U^n) = u^nU^n$. Then for each $u\in \T^d$, $\al_u$ commutes with Fourier multipliers on $\qt$, and $u\mapsto \al_u$ is a strongly continuous family of unitary operators on $L_2(\qt)$. By the definition of convolution, we have
\beq\label{convo-translation}
\psi\ast x  =\int_{\T^d } \psi(u)\, \al_u^{-1}( x)  \, dm(u). \eeq
Since $\al_u$ and $\frac{D_j}{\sqrt{D_1^2 +D_2^2 +\cdots+ D_d^2}}$ commute, we see that $1\otimes \al_u$ commutes with $\sgn(\D)$. Therefore, by the fact that $\big(\al_u^{-1}(x)\big) y = \al_u^{-1} \Big(  x \big(\al_u( y)\big)   \Big)$, we obtain
$$[\sgn(\D), 1\otimes M_{\psi\ast x} ]     =\int _{\T^d}\psi (u)\, (1\otimes \al_{u^{-1}}) [\sgn(\D), 1\otimes M_{x} ] (1\otimes \al_u)   \, dm(u) .$$
Applying \cite[Lemma~18]{LMSZ2017} to the finite Borel measure $\psi(u)\, dm(u)$ on $\T^d$, we get
$$\|[\sgn(\D), 1\otimes M_{\psi\ast x} ]\|_{\mcL_{d,\8}} \leq {\hl C_d}\|[\sgn(\D), 1\otimes M_{x} ]\|_{\mcL_{d,\8}} \|\psi\|_1 $$
{\hl where the constant comes from the use of the quasi-triangle inequality in the $\mcL_{d,\8}$ quasi-norm.}
This now completes the proof.
\end{proof}

%

\begin{proof}[Proof of Theorem \ref{necessity}]
Firstly, we prove the theorem for self-adjoint $x\in {\hl L_2 (\qt)}$. If we show that $x\in \dot{H}_d^1(\qt)$, then Corollary \ref{trace formula-bound} will ensure that there exists a constant $c_d >0 $ such that for all continuous normalised traces $\vf$ on $\mcL_{1,\8}$,
$$c_d \|x\|_{\dot{H}_d^1} \leq  \vf(|\qd x|^d )  ^{\frac 1 d }\leq \|\qd x\|_{d,\8}.$$ 
Thus, we are reduced to proving $x\in \dot{H}_d^1(\qt)$.

Consider the square Fej\'er mean
$$F_N(x)  = \sum_{m\in \ent^d, \max_j |m_j| \leq N}   \Big(1-\frac{ |m_1|}{N+1}\Big) \cdots  \Big(1-\frac{ |m_d|}{N+1}\Big)\widehat{x}(m) U^m .$$
For every $N\in \nat$, it is the convolution of $x$ with the periodic function
$$F_N(u) = \frac{1}{( N+1 )^d } \Big(   \frac{\sin \big(\pi (N+1) u_1\big)}{\sin \big(\pi   u_1\big)}  \Big)^2 \cdots \Big(   \frac{\sin \big(\pi (N+1) u_d\big)}{\sin \big(\pi   u_d\big)}  \Big)^2.$$
The family $\{F_N\}_{N\in \nat}$ is an approximate identity of $L_1(\T^d)$ (see \cite{Graf2008}), so we have uniform bound of $\| F_N\|_1$ in $N\in \nat$. Thus we can apply Lemma \ref{commutator-convo} to $F_N$. The result is
$$\|\qd\big( F_N(x)\big)\|_{\mcL_{d,\8}} \leq \|\qd x\|_{\mcL_{d,\8}} \| F_N\|_1 \leq C \|\qd x \|_{\mcL_{d,\8}}  .$$
Since each $F_N(x)$ is a polynomial in $\qt$, Corollary \ref{trace formula-bound} yields
\be
c_d \|F_N(x)\|_{\dot{H}_d^1}   \leq \vf(|\qd\big(F_N(x)\big)|^d ) ^{\frac 1 d } \leq C\|\qd x\|_{\mcL_{d,\8}} . \ee
Hence, for each $j $, we obtain a bounded sequence $\{ \pd_j F_N(x) \}_{N\in \nat}$ in $L_d(\qt)$. Moreover, since $L_d(\qt)$ is reflexive, we may assume that $\pd_j F_N(x) $ converges to some $y_j \in L_d(\qt)$. On the other hand, by \cite[Proposition~3.1]{CXY2013}, we have $\lim_N F_N(x) = x$ in {\hl$L_2(\qt)$}. Hence, $\pd_j F_N(x)  $ converges to $\pd_j x$ in $\Dist'(\qt)$. Therefore, we have $y_j = \pd_j x  \in L_d(\qt)$. Consequently, we conclude that $x\in \dot{H}_d^1(\qt)$.

It remains to consider $x\in  {\hl L_2 (\qt)}$ which are not self-adjoint. Write $x= x_1 + \ri x _2$ with 
$$x_1 = \frac{1}{2}    (x +x^*),\quad x_2 = \frac{1}{2\ri }  (x-x^*).$$
If $[\sgn(\D), 1\otimes M_{x} ] \in \mcL_{d, \8}$, then $[\sgn(\D), 1\otimes M_{x^*} ]=- [\sgn(\D), 1\otimes M_{x} ]^* \in \mcL_{d,\8}$. Then we have $[\sgn(\D), 1\otimes M_{x_1} ]\in \mcL_{d,\8}$ and $[\sgn(\D), 1\otimes M_{x_2} ]\in \mcL_{d,\8}$. By the above conclusion for self-adjoint elements, we know that $x_1 , x_2\in \dot{H}^1_d(\qt)$, which implies $x= x_1+  \ri x_2 \in\dot{H}^1_d(\qt)$. More precisely,
\be\begin{split}
\|x\|_{\dot{H}^1_d} &\leq \|x_1\|_{\dot{H}^1_d} +\|x_2\|_{\dot{H}^1_d} \\
&\leq C_1(\|[\sgn(\D), 1\otimes M_{x_1} ]\|_{ \mcL_{d,\8}} +\|[\sgn(\D), 1\otimes M_{x_2} ]\|_{ \mcL_{d,\8}} ) \\
& \leq C_2(\|[\sgn(\D), 1\otimes M_{x} ]\|_{ \mcL_{d,\8}} +\|[\sgn(\D), 1\otimes M_{x} ]^*\|_{ \mcL_{d,\8}} )  \\
& = 2C_2 \|[\sgn(\D), 1\otimes M_{x} ]\|_{ \mcL_{d,\8}} .
\end{split}\ee
The theorem is proved.
\end{proof}

{\hl
\noindent{\bf Acknowledgements.} The authors wish to thank the anonymous referees for careful reading and useful suggestions; in particular one referee pointed out how our main results could be proved without an $L_\infty$ condition. We are also greatly indebted to Professor Rapha\"el Ponge for many helpful comments on the section of Pseudodifferential Operators. The authors are supported by Australian Research Council (grant no. FL170100052); X. Xiong is also partially supported by the National Natural Science Foundation of China (grant no. 11301401).
}

\end{document}